%% file: st_article.tex
\documentclass[onefignum,onetabnum]{siamart250211}

\headers{Unfitted Space-Time Methods for Transport Problems}{E. Burman, F. Heimann}

\title{Higher Order Unfitted Space-Time Methods for Transport Problems\thanks{Submitted to the editors DATE.
}}

\author{Erik Burman\thanks{Department of Mathematics, University College London, UK, \texttt{e.burman@ucl.ac.uk}} \and
Fabian Heimann\thanks{Department of Mathematics, University College London, UK, \texttt{f.heimann@ucl.ac.uk}}}

\input{st_shared_siam}

\usepackage{pifont}

\ifpdf
\hypersetup{
  pdftitle={Unfitted-Space-Time Methods for Transport Problems},
  pdfauthor={E. Burman, F. Heimann}
}
\fi

\newcommand{\jump}[1]{[ #1 ]}
\newcommand{\tripplenorm}[1]{||| #1 |||}
\newcommand{\tripplejump}[1]{||[ #1 ]||}

\begin{document}

\maketitle

\begin{abstract}
In this article, we present an Unfitted Space-Time Finite Element method for the scalar transport equation posed on moving domains. We consider the case of the domain boundary being transported by the same velocity field as the scalar concentration inside the physical domain. A standard continuous Galerkin Finite element space is considered on a fixed background mesh, as well as tensor product Space-Time elements, which can be discontinuous along time slice boundaries. For the computational geometry, we opt for a spatially second-order accurate approximation variant in the mathematical analysis. In particular, we establish stability in a problem-specific norm and prove a priori error bounds of high order. Numerical examples illustrate these theoretical findings.
\end{abstract}



\section{Introduction}
Finite Element methods (FEM) provide a powerful framework for the numerical solution of partial differential equations, which appear in a wide range of physical applications, including flow simulations \cite{brennerscott, ern2021finite}. In this paper, we present a particular novel Finite Element method for the transport (or advection) equation posed on a moving domain, where the domain movement follows the flow.

In the time-dependent transport problem, an initial concentration function $u_0(x)$ is assumed to be given on a domain $\Omega(t_0) \subset \mathbb{R}^d, (d=1,2,3)$, which could model for instance a chemical concentration in a droplet. In addition, a transport field $\vec{w}$ is given, and we are interested in solving for a function $u(x,t)$, which describes the concentration after some time $t$ has passed and both species and the overall domain are transported by $\vec{w}$. It can be either divergence-free or feature $\mathrm{div}(\vec{w}) \neq 0$. This physical process can be both seen as a problem of mathematical interest in itself, as well as a relevant model problem for more involved flow simulations, such as Navier-Stokes flows, under certain simplifying assumptions. Note that we also model the limit case of vanishing diffusion of a convection-diffusion problem, which -- in comparison -- entails some simplifications and challenges. The transport equation also bears computational relevance for advecting levelset equations.

Drawing from the guiding example of a droplet deforming in a flow process, we focus on the moving domain setup, i.e. we assume the physical domain boundary to be transported with $\vec{w}$ as well. This poses an interesting challenge for formulating a Finite Element method, as these rely on a computational mesh to yield a discrete solution function space. One strategy for handling this complexity, Unfitted Finite Elements (alternative labels include CutFEM, fictitious domain methods) have been developed in the last decades and center around the idea of generating a fixed background mesh on a suitable domain $\tilde \Omega \subset \mathbb{R}^d$, where $\Omega(t) \subset \tilde \Omega$ for all times $t$ of interest \cite{burman15, Burman_Hansbo_Larson_Zahedi_2025}. This allows for a straightforward handling of the computational domain, if stability is ensured by e.g. appropriate stabilisation bilinear forms, and according methods of numerical integration on the cut elements are available.

The approach in this article combines Continuous Galerkin (CG) Finite Elements in space with 1D elements in time, which are discontinuous along time slice boundaries. In that way, we derive a space-time Finite Element method of high order, where the meshing task is solved on a background domain of spatial dimension, and prismatic elements are (implicitly) considered for space and time. Among the innovations of the method is a space-time variant of the Ghost penalty stabilisation, which will solve both stability issues in relation to small cuts and stabilise the weak form of the transport equation as it relates to the use of CG elements in space.

In relation to the discrete geometries, we make use of a specific feature of the physical problem at hand: Because of the straightforward boundary conditions and the locality of the differential operator, we observe that a slight geometrical error in the computational domain as compared to the physical one is not limiting high order convergence of the overall method, assuming initial data and flow field are known on an appropriate neighbourhood of the exact domain. In the Unfitted FE framework, this observation entails that the use of isoparametric mapped elements is not needed, which simplifies the rigorous mathematical analysis. We include the variant method using discrete higher order geometries and isoparametric finite elements in the presentation and numerical examples to set a computational benchmark and illustrate our construction further. (For details on the particulars of the isoparametric mapping considered, we refer the reader to \cite{L_CMAME_2016, LR_IMAJNA_2018, heimann2024geometrically}.)

As it relates to the mathematical analysis, we establish inf-sup stability of the proposed stabilised bilinear form in a problem specific norm. The stability constant scales with an exponential factor depending on final time $T$ and $|\mathrm{div}(\vec{w})|$. Together with continuity, the stability implies the unique solvability or well-posedness of the discrete problem. Combining this result with approximation bounds and interpolation errors, we obtain overall a priori error bounds of higher order, which is $\mathcal{O}(h^k)$ in $H^1$ type norms for arbitrary polynomial order $k$ of the underlying space, assuming simultaneous refinements in $h$ and $\Delta t$.

To establish these results, the assumption on the extended domain of the exact solution is complemented by regularity assumptions on the levelset function, which allow to construct and analyse the computationally feasible standard $\mathcal{O}(h^2)$ accurate unfitted discrete approximation of the domain based on a piecewise linear levelset function. The discrete domain is assumed to be continuous in time to avoid the necessity of supplementing the calculation of the DG-in-time jump terms by a spatial extension operator. In that way, we are able to relate the discrete bilinear form to a conservative counterpart bilinear form (known in the literature of the case with diffusion \cite{MYRBACK2024117245, heimann2025discretizationerroranalysishigh}) up to a sufficiently small deviation, in order to obtain the symmetric positive contributions in the overall stability argument.

Finally, we study both primal and conservative formulation in numerical experiments and conclude that the geometric undersampling of the $\mathcal{O}(h^2)$ accurate levelset seems particularly suited for the primal formulation, whilst the conservative formulation benefits from higher order geometric accuracy.
\subsection{Related Methods in the literature}
The method at hand relates to the literature on Unfitted Finite Element methods as follows: First, note that relevant techniques in the Unfitted setup had been developed first in the merely spatial case. (Cf. e.g. \cite{burman15, Burman_Hansbo_Larson_Zahedi_2025} and the references therein for an overview.) Later, applications of these to time-dependent problems have been studied as well, for a variety of partial differential equations. It is instructive to distinguish time stepping (often also called Eulerian) methods and space-time methods, where the former are based on a spatial discretisation combined with time stepping methods such as Runge-Kutta or BDF schemes, whilst in the latter case finite elements in space and time are employed. As it relates to time stepping methods, the convection-diffusion problem on unfitted moving domains was studied in \cite{LO_ESAIM_2019}, with higher order generalisations investigated in \cite{LL_ARXIV_2021}. The generalisation to Stokes problems was considered in \cite{IMAJNA_vWRL_2021, burman2020eulerian, CLAUS2019185, OlshanskiivonWahl+2025+961+979}. As for the space-time methods, the convection-diffusion problem was considered in \cite{HLP2022, FH_PHD_2025, heimann2025discretizationerroranalysishigh, MYRBACK2024117245, Z18, HLZ16, BADIA202360}
Computational accounts of Navier-Stokes systems moreover were introduced in \cite{anselmann2021cutfemNSE, FRACHON201977}.

Another relevant area of literature relates to the transport equation posed on fixed domains, both for the time-dependent and time-independent variant. Broadly, many methods follow the Discontinuous Galerkin Finite Element approach for these problems \cite[Chapters 2, 3]{di2011mathematical}, often combined with highly efficient time stepping schemes such as Runge-Kutta. Alternatively, Continuous Galerkin methods including additional stabilisation terms have been developed for related problems. (See e.g. \cite[Chapter 61]{ern2021finite3}, \cite{john2018finite} and the references therein for an overview)
Among the developed stabilisation options are the Streamline upwind/Petrov-Galerkin (SUPG) framework \cite{BROOKS1982199, JOHNSON1984285}, or the Gradient-Jump- or Continuous Interior penalty method \cite{burman2007continuous, burman2007continuous_nse}.
The particular time-dependent Ghost penalty stabilisation employed in this article can be understood as a member of this broader class of stabilisation variants, now applied to the moving space-time setting. See also \cite{fu2024bound} for a related Unfitted DG method for one dimensional hyperbolic conservation laws. In \cite{doi:10.1137/24M1674182, Schwering:1024750}, a DG-in-space discretisation of the transport equation for the application of levelset advection on narrow bands was introduced and integrated into a moving domain flow solver.

In relation to the literature background on Ghost penalty stabilisations, we mention the initial work \cite{burman2010ghost}. In the space-time context, a variant of this stabilisation called direct space-time Ghost penalty was developed in \cite{preuss18, heimann20, HLP2022,MYRBACK2024117245, heimann2025discretizationerroranalysishigh}.
In this article, we adapt the variant of \cite{heimann2025discretizationerroranalysishigh, FH_PHD_2025} to the case without diffusion and with the dual role of stabilising CG for transport, as well as yielding discrete inequalities on cut meshes relevant for stability.

Related discretisation options within the wider set of Unfitted Finite Elements include the $\phi$-FEM paradigm, where the geometry representation and the physical problem are coupled as an alternative to cut integration \cite{doi:10.1137/19M1248947}. Moreover in shifted boundary methods, the boundary conditions are replaced by corrected data on an intentionally shifted surrogate boundary, which is similar in structure to our computational approach to the boundary \cite{MAIN2018972, SONG201845}. The aggregated Finite Element Method's merging of neighbouring otherwise small-cut elements can be regarded as another alternative to the Ghost-penalty stabilisation approach \cite{BADIA2018533, doi:10.1137/18M1185624, BADIA202360, BADIA2022114232}. Furthermore, the alternative Finite Cell Method paradigm relies on highly adaptive numerical integration rules on fictitious domains \cite{schillinger2015finite}.

Beyond these geometry and stabilisation paradigms, the transport limit was also studied in finite volume / DG cut-cell methods involving Domain-of-Dependence stabilisations \cite{doi:10.1137/19M1268318, MAY2022126854, 10.1007/978-3-031-20432-6_33}, which could be regarded as a counterpart to the Ghost-penalty in terms of avoiding small cut cell issues. Another related option from the finite-volume framework on Cartesian grids is the $h$-box method \cite{doi:10.1137/110829398}.

For problems involving only deformations and no topology changes, a relevant alternative to the Unfitted or CutFEM background mesh approach consists in body-fitted meshes or Arbitrary Lagrangian Eulerian approaches \cite{HIRT1974227, TEZDUYAR1992339, doi:10.1137/120862715}. If topology changes are ruled out, the tradeoff big-picture wise is between mesh quality maintenance and cut-stabilisation.
\subsection{Structure of the article}
The presentation in the remainder of the article proceeds as follows. In \Cref{sect_method_intro}, we start by putting forward the details of the proposed numerical method. This includes the detailed description of the model problem, which motivates the definition of the bilinear and linear forms. Both uncurved and curved mesh variant methods are introduced. Next, in \Cref{sect_stab_analysis} we establish inf-sup stability of the uncurved mesh method. Combining this with interpolation results in \Cref{sect_a_pr_est} yields higher order a priori error bounds. In \Cref{sect_num_exp}, numerical examples are provided.
\section{Method Introduction} \label{sect_method_intro} Before introducing the details of the method, we present our problem of consideration.
\subsection{Model problem} We are interested in solving the transport or advection equation on moving domains, whose boundary is advected by the same velocity. In particular, let a time interval of interest $t \in [0,T]$ be given, as well as a bounded moving domain $\Omega(t) \subset \mathbb{R}^d, d=2,3$, a convection field $\vec{w}(x,t)$, $x \in \Omega(t), t \in [0,T]$, and initial concentration data $u_0 (x), x \in \Omega(0)$.\footnote{Technicalities about e.g. domains of definition for $\vec{w}$, $u_0$ will be given later.} The aim is to model the time evolution of a concentration $u(x,t)$ for $x \in \Omega(t), t\in [0,T]$, where both concentration values at fixed points as well as the domain boundary $\partial \Omega$ are assumed to move with $\vec{w}$. This function is given as the solution to the following partial differential equation in strong form: Find $u= u(x,t)$ such that
\begin{align}
 \partial_t u + \mathrm{div}(\vec{w} u) &= f &&\mathrm{on} \quad \Omega(t)\quad \forall t \in [0,T], \label{strong_form_problem} \\
 u(x,0) &= u_0(x) &&\mathrm{on} \quad \Omega(0), \nonumber
\end{align}
where $f=f(x,t)$ represents a right-hand side function modeling sources or sinks. Moreover, we assume that $\Omega(t)$, or $\partial \Omega(t)$, is moving in time with $\vec{w}$ as well. Note that this implies that the formulation \Cref{strong_form_problem} is well posed insofar as no part of $\partial \Omega(t)$ would need inflow (e.g. Dirichlet) data.\footnote{Compare e.g. to the formulation \cite[Eq. (3.1)]{di2011mathematical} and note that for determining the inflow part, the inner product with the relative convection velocity, which vanishes here, needs to be considered. An according contribution can be added to bilinear and linear form of our method as well to consider the case of inflow fields / non-matched boundary evolution.}

This particular feature of the problem, needing no boundary data, entails the following observation, which is relevant for motivating our use of a discrete shifted geometry.
\begin{lemma}[Extending domain for moving domain transport problem] \label{lemma_extending_domain}
 Assume a time-dependent domain $\Omega^e(t)$ is given, which has
 \begin{equation}
  \Omega(t) \subseteq \Omega^e(t) \quad \forall t \in [0,T],
 \end{equation}
 together with an extended transport velocity field, $\vec{w}^e$, defined on $\Omega^e(t)$ with $\vec{w}^e(\cdot,t)|_{\Omega(t)} = \vec{w}(\cdot,t)$, an initial data function $u_0^e$ defined on $\Omega^e(0)$ so that $u_0^e|_{\Omega(0)} = u_0$, and an extended right-hand side $f^e$ defined on $\Omega^e(t)$ with $f^e(\cdot,t)|_{\Omega(t)} = f(\cdot,t)$. Further assume $\partial \Omega^e(t)$ is moving with $\vec{w}^e$, and let $u^e$ be defined as the solution to the extended problem
 \begin{align}
 \partial_t u^e + \mathrm{div}(\vec{w}^e u^e) &= f^e &&\mathrm{on} \quad \Omega^e(t)\quad \forall t \in [0,T], \label{strong_form_problem_ext} \\
 u^e(x,0) &= u_0^e(x) &&\mathrm{on} \quad \Omega^e(0), \nonumber
\end{align}
Then, denoting by $u$ the solution to \Cref{strong_form_problem}, it holds
\begin{equation}
 u^e (\cdot, t) |_{\Omega(t)} = u(\cdot, t) \quad \forall t \in [0,T].
\end{equation}
\end{lemma}
\begin{proof}
 The restriction of $u^e$ satisfies \Cref{strong_form_problem}.
\end{proof}

Let us introduce some notations. First, we assume that the time-dependent domain is given implicitly by a levelset function $\phi$:
\begin{assumption}[Levelset function] \label{lset_assumption}
 There exists a scalar valued function $\phi$ so that for all $t \in [0,T]$
 \begin{equation}
  \Omega(t) = \left\{ x \in \mathbb{R}^d \, | \, \phi(x,t) < 0 \right\}.
 \end{equation}
 Moreover, we assume $\phi \in C^2(U_\Gamma)$ for some neighbourhood $U_\Gamma$ of the boundary $\{ (x,t) \in \mathbb{R}^d \times [0,T] \, | \, \phi(x,t) =0 \}$, and that $\phi$ is non-degenerate, i.e. $\| \nabla \phi \| \geq c_\phi > 0$ on $U_\Gamma$.
\end{assumption}
We define the space-time domain as follows
\begin{definition}[Space-time domain]
 The space-time domain $Q \subset \mathbb{R}^{d+1}$ is defined as
 \begin{equation}
  Q := \bigcup_{t \in [0,T]} \Omega(t) \times \{ t \}.
 \end{equation}
\end{definition}
Moreover, we define a spatial neighbourhood (or extension by a neighbourhood) of $Q$ for a width parameter $\delta > 0$ as follows:
\begin{equation}
 U_\delta (Q) := \{ (x,t) \in \mathbb{R}^d \times [0,T] \, | \, \mathrm{dist}_{\mathbb{R}^d} (x,\Omega(t))<\delta \}.
\end{equation}
In the light of these notations and motivated by the observation \Cref{lemma_extending_domain}, we make the following assumption on the existence of a strong form solution.
\begin{assumption} \label{u_ext_assumption}
 Assume there exists some $\delta > 0$ and a set $U \subset \mathbb{R}^d \times [0,T]$, $U \supseteq U_\delta$, together with functions $\vec{w}^e$, and $f^e$ defined on $U$, which extend $\vec{w}$ and $f$ on $Q$ respectively, and a function $u_0^e$, which extends $u_0$ on $\Omega(0)$. We assume $\vec{w}^e \in W^{1,\infty}(U)$. Moreover, assume that for all $t \in [0,T]$, the spatial boundary $\partial_s U$ moves with $\vec{w}^e$.\footnote{This part of the assumption is the reason why we allowed for $U \supseteq U_\delta$.}
 
 We assume that there exists a solution $u^e \in H^{\mathrm{reg}(u)}(U)$ to the strong form extended problem \Cref{strong_form_problem_ext}, which is bounded as
 \begin{equation}
  \| u^e \|_{H^{\mathrm{reg}(u)}(U)} \lesssim \| u^e \|_{H^{\mathrm{reg}(u)}(Q)}. \label{eq_bnd_u_e}
 \end{equation}
\end{assumption}
Here and throughout the article, $\| \cdot \|_{H^k(Q)}$ denotes the usual Sobolev space norm of order $k$. By $\mathrm{reg}(u)$, we denote an integer parameter for the regularity of $u$ in the space-time Sobolev space manner, to be specified later in the final estimate. Moreover, we use the symbol $A \lesssim B$ to indicate the existence of a constant $C$ independent of $h$, $\Delta t$ and the position of the boundary relative to the mesh, but possibly depending on the fixed polynomial degrees, the stabilisation parameter, the regularity and non-degeneracy bounds of the levelset function, bounds on $\vec{w}$ and the shape regularity of the mesh, so that $A \leq C \cdot B$.\footnote{As a further remark on notation, we will only occasionally denote the extended functions $f^e$, $\vec{w}^e$, $u_0^e$ with their $^e$ index when we want to specifically highlight this aspect. For improved overall readability, whenever any of the three functions is denoted in their non-extended version such as $u_0$ on a larger domain, it is understood that strictly speaking $u_0^e$ is meant.} By construction, \Cref{lemma_extending_domain} implies that $u^e$ solves also \Cref{strong_form_problem}.
\begin{remark} \label{remark_extension_width}
  Assumptions \ref{lset_assumption} and \ref{u_ext_assumption} interact as follows in the upcoming analysis: The insight \Cref{lemma_extending_domain} allows us to geometrically undersample the computational geometry, which comes at the theoretical cost of needing an extended solution, i.e. an assumption of the type of Assumption \ref{u_ext_assumption}. Motivated by its computational convenience and salience with Unfitted FEM, we opt for the piecewise linear levelset geometry, which pairs naturally with the regularity in Assumption \ref{lset_assumption}. The maximal geometric error incurred by this approximation in a pointwise distance sense is of order $\mathcal{O}(h^2)$. However, the analysis requires the extended data and solution on the active mesh $\mathcal{E}(Q^{\mathrm{lin},n})$ introduced below, which consists of full space-time elements and hence extends beyond $Q$ by $\mathcal{O}(h + \Delta t)$. Thus, any width of the extension $\delta$ such as in Assumption \ref{u_ext_assumption} with $\delta \gtrsim h + \Delta t$ is appropriate, in particular our choice of a fixed $\delta > 0$ independent of $h$ and $\Delta t$.
\end{remark}
\subsection{Discrete Method involving uncurved elements}
We turn our attention now to formulating the discrete version of this problem. First note that in the Unfitted FE paradigm, the following mesh assumption is salient.
\begin{assumption}[Unfitted meshes and discrete levelset function]
 Assume there exists a fixed-in-time polygonal background domain $\tilde \Omega \subset \mathbb{R}^d$, so that $\forall t \in [0,T] \ \Omega(t) \subset \tilde \Omega$. Let $\{ \mathcal{T}_h \}_h$ be a family of quasi-uniform, shape-regular simplicial meshes with maximal element size $h$ on $\tilde \Omega$. Denote the interior facets of the mesh as $\mathcal{F}_h^i$.
 
 In relation to the time discretisation, we assume that the time interval $[0,T]$ is subdivided into elements or time slices $I_n$ of equal width as
 \begin{equation}
  I_n = [t_{n-1}, t_n], \quad \Delta t = t_n - t_{n-1} \quad n=1,\dots,N, \quad t_0 = 0, t_N = T.
 \end{equation}

 We assume that the exact levelset function $\phi$ is extended continuously to all of $\tilde \Omega \times [0,T]$, where the $C^2$ regularity and non-degeneracy of Assumption \ref{lset_assumption} are only required on the neighbourhood $U_\Gamma$ of the boundary. In particular, no smoothness is needed away from $U_\Gamma$; e.g. the distance-type levelset functions in \Cref{sect_num_exp} are not smooth at the domain centre. The discrete geometry is implicitly described in a related manner by an approximate levelset function $\phi^{\mathrm{lin}}$. We are specifically interested in using the setup of \cite{heimann2024geometrically} for $\phi^{\mathrm{lin}}$, which relies on a temporally discrete approximation of $\phi$, called $\phi_h$. It is of the order $q_t$ in relation to the 1D finite elements in time on $I_n$. In this article, we assume $q_t\geq 1$. (And give a computational side remark on $q_t = 0$ in Remark \ref{remark_on_qt0}.) Next, $\phi^{\mathrm{lin}}$ is defined for any $t \in [0,T]$ as an elementwise (in relation to the mesh $T \in \mathcal{T}_h$) linear interpolation of $\phi_h$. This implies, based on the assumption $\phi \in C^2(U_\Gamma)$ and $q_t \geq 1$ (see \cite{heimann2024geometrically} for further details\footnote{Formally speaking, the framework of \cite{heimann2024geometrically} would require $\phi \in C^3$ for $q_s = q_t = 1$. However, this assumption is only fully needed for interpolation bounds in the general setup of $\Delta t, h$ decoupled. As our interpolation operators do not interact with the isoparametric mapping, we can obtain the relevant results (cf. proof of \cref{lemma_B_Bmc_diff}), the existence of a mapping between exact and discrete geometries, its closeness to identity, and a bound on the discrete normals derived from \Cref{eq_phi_diff}, with $\phi \in C^2$.}), for all space-time elements $T \times I_n \subseteq U_\Gamma$,
 \begin{equation}
  \| \phi - \phi^\mathrm{lin} \|_{\infty, T \times I_n} \lesssim h^2 + \Delta t^{2}, \quad \| \nabla_{x,t} (\phi - \phi^\mathrm{lin}) \|_{\infty, T \times I_n} \lesssim h + \Delta t. \label{eq_phi_diff}
 \end{equation}
 Here, $\nabla_{x,t}$ denotes the space-time gradient, taken elementwise for $\phi^{\mathrm{lin}}$; the bound on the temporal derivative enters through the space-time normals in the proof of \Cref{lemma_B_Bmc_diff}. The powers of $\Delta t$ reflect the $C^2$ regularity; for smoother $\phi$, the temporal interpolation of order $q_t$ yields higher powers of $\Delta t$, which we do not exploit. Moreover, by uniform continuity of $\phi$ and $|\phi| > 0$ outside $U_\Gamma$, for $h$ and $\Delta t$ sufficiently small every space-time element intersecting the zero level set of $\phi$ or of $\phi^{\mathrm{lin}}$ lies in $U_\Gamma$, so that \Cref{eq_phi_diff} covers all elements relevant for the discrete geometry.

 Furthermore, $\phi_h$, and $\phi^\mathrm{lin}$, shall be continuous in time. Being able to realise this in the general moving domain setup is one reason why we choose (at least) $q_t = 1$.
\end{assumption}
The spatially element-wise linear approximation $\phi^\mathrm{lin}$ is common in Unfitted Finite Element Methods as it allows for straightforward numerical integration \cite{burman15, L_PHD_2015}. Specifically, elementwise straight cuts allow for a decomposition of the computational integration domain into few simplices. On these simplices, it is straightforward to map standard Gaussian integration rules. Relating this to the transport problem without essentially any boundary conditions, the piecewise linear discrete geometry presents a good discrete geometry candidate by deviating at most $h^2$ far and allowing for the mentioned straightforward numerical integration procedure.

Based on $\phi^\mathrm{lin}$, we introduce the following notation for discrete space- and space-time domains:
\begin{align}
 \Omega^{\text{lin}}(t) :&= \{ x \in \tilde \Omega \, | \, \phi^{\text{lin}} (x,t) < 0 \}, \quad Q^{\text{lin},n} := \bigcup_{t \in I_n} \Omega^{\text{lin}}(t) \times \{ t \}.
\end{align}
We denote the extended domain, which is the collection of all elements containing some part of the discrete domain, on a time slice $I_n$ as
\begin{align}
 \mathcal{E}(Q^{\text{lin},n}) :&= \bigcup \{ T \times I_n \, | \, (T \times I_n) \cap Q^{\text{lin},n} \neq \varnothing, T \in \mathcal{T}_h \} .
\end{align}
Note that because of the asymptotic closeness of $\phi$ and $\phi^{\text{lin}}$, for some sufficiently small $h$ and $\Delta t$, the domains $Q^{\text{lin},n}$ and $\mathcal{E}(Q^{\text{lin},n})$ will be contained in $U$ as in Assumption \ref{u_ext_assumption}. In the following, we assume that $h$ and $\Delta t$ are always sufficiently small for this set inclusion to hold.

As for the discrete spaces, first assume a spatial polynomial order $k_s \geq 1$ and a temporal polynomial order $k_t \geq 1$ are given, which we will sometimes abbreviate as $k = (k_s,k_t)$. Then, the standard Continuous Galerkin and Discontinuous Galerkin spaces on $\tilde \Omega$ are given as
\begin{align}
 V_h^{k_s} &:= \{ v \in H^1(\tilde \Omega) \, | \, v|_T \in \mathcal{P}^{k_s}(T) \ \forall T \in \tilde{\mathcal{T}}_h \}, \\
 V_{\text{DG}}^{k_s} &:= \{ v \in L^2(\tilde \Omega) \, | \, v|_T \in \mathcal{P}^{k_s}(T) \ \forall T \in \tilde{\mathcal{T}}_h \}.
\end{align}
Next, we define the corresponding tensor-product space-time spaces regarding the time interval $I_n$ as
\begin{align}
 W_h^{n,(k_s,k_t)} := V_h^{k_s} \otimes \mathcal{P}^{k_t}( [t_{n-1},t_n]), \label{discrete_space_full_backgr_dom} \quad
 W_{\text{DG}}^{n,(k_s,k_t)} := V_{\text{DG}}^{k_s} \otimes \mathcal{P}^{k_t}( [t_{n-1},t_n]).
\end{align}
From the first space, we want to activate those degrees of freedom, which are related to the domain $\mathcal{E}(Q^{\text{lin},n})$ of the respective time slice in order to save unnecessary degrees of freedom for computational reasons. This motivates to define for $\mathcal{Q}$ a space-time domain such as $\mathcal{E}(Q^{\text{lin},n})$
\begin{equation}
 W_{\text{cut}}^{n}(\mathcal{Q}) = \{ v \in W_h^{n,k} \, | \, v \textnormal{ vanishes outside of } \mathcal{Q} \textnormal { as element of } W_h^{n,k}\}.\footnote{By 'vanishes outside of $\mathcal{Q}$' we mean, in terms of a Lagrange basis, that degrees of freedom at nodes outside $\overline{\mathcal{Q}}$ are set to zero while those on $\partial\mathcal{Q}$ remain free. Hence the function is in general not pointwise zero on the adjacent layer of elements.}
\end{equation}
In terms of functions for all time slices, we define in line with the Discontinuous Galerkin in time paradigm
\begin{align}
 W_h &= \{ v \in L^2(\bigcup_{n=1}^N \mathcal{E}(Q^{\text{lin},n})) \, | \, \forall n=1,\dots, N: v|_{\mathcal{E}(Q^{\text{lin},n})} \in W_{\text{cut}}^{n}(\mathcal{E}(Q^{\text{lin},n})) \}
\end{align}
For defining the linear and bilinear forms of our method, we will use discrete geometries, which are indicated by an upper index $h$. As we have opted to use the piecewise linear discrete geometries for our main analysis, we define
\begin{equation}
 \Omega^h(t) := \Omega^{\text{lin}}(t), \quad Q^{h,n} := Q^{\text{lin},n}, \quad \mathcal{E}(Q^{h,n}) := \mathcal{E}(Q^{\text{lin},n}). \label{discretehdomains_def_as_lin}
\end{equation}
Moreover, the global space-time variants should be defined as
\begin{equation}
 Q^h = \bigcup_{n =1}^N Q^{h,n}, \quad \mathcal{E}(Q^h) = \bigcup_{n =1}^N \mathcal{E}(Q^{h,n}).
\end{equation}
In terms of these domains, we now introduce the bilinear form of the problem, which contains the physical summand stemming from testing with $v$ and integrating over the domain, and upwind jump penalties as far as DG across time slice boundaries is concerned.
 \begin{align}
   B_h(u,v) = &\sum_{n=1}^N (\partial_t u + \mathrm{div} (\vec{w} u), v)_{Q^{h,n}}
  + \sum_{n=1}^{N-1} (\jump{u}^n, v^n_+)_{\Omega^{h}(t_n)} + (u^0_+, v^0_+)_{\Omega^{h}(t_0)}. \label{eq_def_Bh}
 \end{align}
 Here and in the following,
 $
  \jump{u}^n := \lim_{s \searrow t_n} u(s) - \lim_{s \nearrow t_n} u(s).
 $
 Note that the well-posedness of the jump relies on the fact that the discrete domain is continuous, by virtue of the continuity of $\phi^{\text{lin}}$.

Accordingly, the right-hand side $f$ is given as
 \begin{equation}
  f_h(v) := \sum_{n=1}^N (f,v)_{Q^{h,n}} + (u_0, v_+^0)_{\Omega^h(t_0)}. \label{eq_def_fh}
 \end{equation}
The bilinear form $B_h$ is complemented by a Ghost-penalty stabilisation bilinear form. We use a space-time version of the direct Ghost penalty as in \cite{preuss18, heimann20, HLP2022, heimann2025discretizationerroranalysishigh}. It will operate based on interior facets of the mesh $F \in \mathcal{F}_h^i$. Fix any such $F \in \mathcal{F}_h^i$ and let $T_1 \in \mathcal{T}_h$ and $T_2 \in \mathcal{T}_h$ be those elements such that $F = \overline{T_1} \cap \overline{T_2}$. We denote the union of those elements as the facet patch $\omega_F = T_1 \cup T_2$. On this facet patch, we define a jump operator for elementwise polynomials as
\begin{equation}
 \jump{u}_{\omega_F} |_{T_i} (x,t) = u(x,t) - \mathcal{E}^p (u|_{T_{3-i}})(x,t), \quad i=1,2,
\end{equation}
where $\mathcal{E}^p$ denotes the polynomial extension.

Based on this notation, we define the stabilisation bilinear form contribution from $F$ as
 \begin{align}
 j_F^n(u,v) :&= \int_{t_{n-1}}^{t_n} \int_{\omega_F} \gamma_J h^j \jump{u}_{\omega_F} \jump{v}_{\omega_F} \mathrm{d}x \mathrm{d}t,
\end{align}
where $\gamma_J > 0$ denotes a scalar valued stabilisation parameter and $j$ the integer scaling of the stabilisation. We choose $j= -1$ for this article. Throughout the analysis, $\gamma_J$ is fixed independently of $h$ and $\Delta t$, as are the polynomial degrees; the constants in the discrete inequalities below may depend on $\gamma_J^{-1}$, so that no uniformity for $\gamma_J \to 0$ is claimed. The choice $j=-1$ is dictated by the first order transport operator rather than being a tuning parameter: It is the scaling for which the Ghost penalty both controls the facet jump terms arising from the CG-in-space discretisation of the transport equation (see \Cref{lemma_facet_jump_gp}) and enters the discrete inverse estimates of \Cref{subsect_gp_estimates} with the $h$-scaling needed in the stability analysis, whereas in the presence of physical diffusion the scaling $j=-2$ is the appropriate one, see \Cref{remark_gp_conv_diff}. The stabilisation shall operate on all facets inside the extended domain, i.e.
\begin{align}
  \mathcal{F}^{n} = \{F \! \in\! \mathcal{F} \textnormal{ s.t. } \exists T_1 \neq T_2,& T_1 \times I_n, T_2 \times I_n \! \subseteq\! \mathcal{E}(Q^{\text{lin},n}) \textnormal{ with } F = T_1 \! \cap\! T_2 \}.
 \end{align}
 Summing over all these facets and all time slices, we obtain the overall stabilisation bilinear form
\begin{align}
 j^n_h(u,v) &= \sum_{F \in \mathcal{F}^{n}} j_F^n(u,v), \quad J(u,v) = \sum_{n=1}^N j^n_h(u,v).
\end{align}
\begin{remark}[Comparison to Ghost penalty in Convection-Diffusion discretisations] \label{remark_gp_conv_diff}
 The proposed variant of the direct Ghost penalty is similar to the variants considered in \cite{preuss18, heimann20, heimann2025discretizationerroranalysishigh, FH_PHD_2025} for convection-diffusion problems. We highlight the two relevant differences. First, in the presence of physical diffusion, the scaling $j$ needs to be chosen as $j=-2$. Second, the set of facets, where the Ghost penalty is applied, $\mathcal{F}^{n}$ is chosen larger in our case to contain both the physical interior and the boundary. This is due to the dual role of the Ghost penalty to stabilise small cut elements and allow for CG elements in space with the transport equation. Because in the convection-diffusion applications the only role of the Ghost penalty is the first one, choosing facets relating to the cut boundary elements suffices in that case.
\end{remark}
We can now pose the discrete problem: Find $u_h \in W_h$ such that for all $v_h \in W_h$
\begin{equation}
 B_h(u_h,v_h) + J(u_h,v_h) = f_h(v_h). \label{eq_discrete_problem}
\end{equation}
 \subsection{Discrete Method on higher order geometries} \label{discrete_curved_geom_method}
 The discrete geometries $\Omega^h(t)$, $Q^{h,n}$ as defined above have second order geometrical accuracy. This order of accuracy can be increased to higher order by an isoparametric mapping, see e.g. \cite{heimann2024geometrically, heimann2025discretizationerroranalysishigh}. When diffusion is involved with according boundary conditions, or initial data or right hand side cannot be extended accordingly, this will be needed to achieve a method of overall high order convergence order. We discuss the according method for the transport equation here as well for computational comparison.\footnote{Comparing the rigorous mathematical results of e.g. \cite{FH_PHD_2025, heimann2025discretizationerroranalysishigh} to the material at hand, we regard the detailed mathematical analysis of the isoparametric method as beyond the scope of this article and leave it as an open question for future research.}

 Let an isoparametric mapping in space and time be given, $\Theta_h(\cdot,t)\colon \tilde \Omega \to \tilde \Omega$, and denote by $\Theta_h^\mathrm{st}$ the space-time version, $\Theta_h^\mathrm{st} (x,t) := (\Theta_h(x,t) ,t)^T$.

 In terms of this, we can define the curved discrete regions as
 \begin{equation}
 \Omega^h(t) := \Theta_h(\Omega^{\text{lin}}(t),t), \quad Q^{h,n} := \Theta_h^{\text{st}}(Q^{\text{lin},n}),
 \end{equation}
 which shall be understood to replace \Cref{discretehdomains_def_as_lin} for this subsection. Then, the discrete space $W_h$ is concatenated with the inverse of the mapping (see e.g. \cite{HLP2022, heimann2025discretizationerroranalysishigh}), and bilinear form $B_h$ and linear form $f_h$ are given as in \Cref{eq_def_Bh}, \Cref{eq_def_fh}, but on the changed higher order discrete domains. Furthermore, the Ghost penalty would be reformulated on curved patches, involving a changed jump operator (see \cite{heimann2025discretizationerroranalysishigh, heimann20} for details).

 With these changes in underlying definitions, the discrete problem for the curved meshes / isoparametric method reads the same as \Cref{eq_discrete_problem}.
\section{Discrete Stability and Continuity analysis} \label{sect_stab_analysis}
In this section, we provide a stability analysis of the formulation \Cref{eq_discrete_problem}, showing inf-sup stability in particular. This implies the unique solvability of the discrete problem.

To this purpose, we introduce the following discretisation related norms
\begin{align}
 \tripplejump{u}^2 &:= \sum_{n=1}^{N-1} (\jump{u}^n, \jump{u}^n)_{\Omega^{h}(t_n)} + (u_+^0, u_+^0)_{\Omega^{h}(0)} + (u_-^N, u_-^N)_{\Omega^{h}(T)} \\
 \tripplenorm{u}^2 &:= h \| \partial_t u + \vec{w} \nabla u \|_{Q^h}^2 + \| u \|_{Q^{h}}^2 + \tripplejump{u}^2 \label{eq_def_tripplenorm}\\
 \tripplenorm{u}^2_j &:= \tripplenorm{u}^2 + J(u,u), \quad \tripplenorm{u}_J^2 := J(u,u).
\end{align}
In the rigorous mathematical analysis, we are interested in the limit behaviour of the method for small mesh sizes $h$ and time steps $\Delta t$. In particular, we are interested in simultaneous refinements. We formally assume
\begin{assumption}[Refinement in $h$ and $\Delta t$] \label{assume_h_delta_t_equiv}
 We assume $h$ and $\Delta t$ to be equivalent up to a constant, i.e.
 \begin{equation}
  h \lesssim \Delta t \quad \textnormal{ and } \quad \Delta t \lesssim h.
 \end{equation}
 In particular, $\Delta t \leq c_t h$ for a constant $c_t$ independent of $h$ and $\Delta t$.
\end{assumption}
\begin{remark}[Comparison to problem-specific norm in Convection-Diffusion discretisations]
 We compare the definition of the problem-specific norm $\tripplenorm{\cdot}$ in this article with the corresponding convection-diffusion analysis presented in \cite{heimann2025discretizationerroranalysishigh, FH_PHD_2025, preuss18}. First, the jump summands $\tripplejump{\cdot}$ are identical. Next, we have added a volumetric $L^2$ summand $\| u \|_{Q^{h,n}}^2$ for the transport equation case. This is to balance contributions from $\mathrm{div}(\vec{w}) \neq 0$, at the cost of a constant scaling with $e^{-\alpha T}$ in the stability result. Moreover, this construction allows us to directly control boundary error terms, where in the case of e.g. \cite{FH_PHD_2025} a more involved trace inequality, using $K \cdot I(\cdot, \cdot)$ summands was needed. Finally, note that in relation to derivatives, we only aim at obtaining control in the material derivative norm in the transport equation, whilst the convection-diffusion analysis was based on separate full $\partial_t u$ and $\nabla u$ summands. In this context, Assumption \ref{assume_h_delta_t_equiv} also fits in well, yielding a uniform scaling with $\Delta t$ (or equivalently: $h$), where in the convection-diffusion case allowing for more general refinements, each of the summands $\partial_t u$ and $\nabla u$ was scaled with $\Delta t$ or $h$ respectively.
\end{remark}
\begin{remark}[Scaling of material derivative in \Cref{eq_def_tripplenorm}]
 In \Cref{eq_def_tripplenorm}, we have opted for a scaling of the material derivative summand of $h$. In the light of Ass. \ref{assume_h_delta_t_equiv}, let us mention that this choice would be equivalent to $\Delta t$ or combinations such as $\min \{ \Delta t, \frac{h}{\| \vec{w} \|_\infty} \}$. Depending on different viewpoints any of these might be regarded most natural.
\end{remark}
\subsection{Discrete Projection Operators}
The stability argument we want to put forward involves, for a given discrete function $u \in W_h$, a rescaled version of the space-time discrete function by a smooth scalar function in time, and derivative operators applied to $u$.\footnote{Details of this can be found below in the proof of \Cref{thm_stability}.} Both of these constructions will not yield a discrete function by themselves. The purpose of the following operators is to remedy this by supplying an appropriate discrete function, together with bounds on the difference incurred.

We start with the case of rescaling of the function $u \in W_h$ by a scalar smooth function $\varphi$ in time. To yield an appropriate discrete function, we want to apply an $L^2$ interpolation on each time slice $I_n$ and summarise well-known properties of this operator as follows.
\begin{lemma}
 For each $I_n$, there exists an interpolation operator, the $L^2$ interpolation, $\Pi^{k_t,n}\colon L^2(I_n ; V_h^{k_s})\to W_{\text{cut}}^{n}(\mathcal{E}(Q^{\text{lin},n}))$,
 so that
 \begin{align}
  \| \Pi^{k_t,n} u \|_{\mathcal{E}(Q^{\text{lin},n})} &\lesssim \| u\|_{\mathcal{E}(Q^{\text{lin},n})}, \quad \|\partial_t \Pi^{k_t,n} u \|_{\mathcal{E}(Q^{\text{lin},n})} \lesssim \| \partial_t u\|_{\mathcal{E}(Q^{\text{lin},n})}.
 \end{align}
 The second estimate requires $u \in H^1(I_n; V_h^{k_s})$. As the operator acts componentwise in time, it is defined verbatim on $L^2(I_n; V)$ for any finite-dimensional spatial space $V$, in particular on spatially discontinuous and vector-valued functions such as $\nabla u$. Furthermore, $\Pi^{k_t,n}$ commutes with the spatial derivative, i.e. $\forall u \in L^2(I_n; V_h^{k_s})$,
 \begin{equation}
  \nabla ( \Pi^{k_t,n} u ) = \Pi^{k_t,n} (\nabla u). \label{eq_Pikt_nabla_commute}
 \end{equation}
\end{lemma}
For a detailed proof of these results, we refer the reader to \cite[Lemma 3.15]{preuss18} / \cite[Appendix A.1]{burman_preuss_25}.

We will apply a superconvergence result, which yields a special deviation bound for our argument class of products between a discrete function and a smooth function $\varphi(x,t) = \varphi(t)$. The relevant property is called discrete commutator property. Exploiting the tensor-product structure of $\mathcal{E}(Q^{\text{lin},n})$, we can apply the results from e.g. \cite{BERTOLUZZA19991097} for a fixed point in space on the respective time element.
\begin{lemma}
 Let $\varphi \in C^{\infty}([0,T]), u \in W_{\text{cut}}^{n}(\mathcal{E}(Q^{\text{lin},n}))$, then
 \begin{align}
  \| \varphi u - \Pi^{k_t,n}( \varphi u)\|_{\mathcal{E}(Q^{\text{lin},n})} &\leq C(\varphi) \cdot \Delta t \cdot \| u \|_{\mathcal{E}(Q^{\text{lin},n})}, \label{eq_superinterpol} \\
  \| \partial_t (\varphi u - \Pi^{k_t,n}( \varphi u))\|_{\mathcal{E}(Q^{\text{lin},n})} &\leq C(\varphi) \cdot \Delta t \cdot ( \| \partial_t u \|_{\mathcal{E}(Q^{\text{lin},n})} + \| u \|_{\mathcal{E}(Q^{\text{lin},n})}). \label{eq_superinterpol_dt}
 \end{align}
\end{lemma}
\begin{proof}
 This is the discrete commutator property \cite[Thm.\ 2.1]{BERTOLUZZA19991097} for $S_h = \mathcal{P}^{k_t}(I_n)$, $k_t \geq 1$, with the $L^2$ projection, applied for each fixed $x$ with $(n,r) = (0,0)$ and $(1,1)$; note that its right-hand side carries the full $H^r$ norm, which yields the term $\| u \|$ in \cref{eq_superinterpol_dt}. Directly: the projection errors are bounded by $\Delta t^{k_t+1} \| \partial_t^{k_t+1} (\varphi u) \|$ and $\Delta t^{k_t} \| \partial_t^{k_t+1} (\varphi u) \|$, and since $\partial_t^{k_t+1} u = 0$, the Leibniz rule and the inverse inequality in time give $\| \partial_t^{k_t+1} (\varphi u) \| \lesssim \Delta t^{-k_t} \| u \|$ and $\lesssim \Delta t^{1-k_t} ( \| \partial_t u \| + \Delta t^{k_t - 1} \| u \| )$, respectively.
\end{proof}
A global variant of this operator is defined as each $\Pi^{k_t,n}$ on the respective $I_n$.
\begin{definition}
 The operator $\Pi^{k_t}$ is defined for arguments $u\in L^2( [0,T]; V_h^{k_s})$ to yield a discrete function $\Pi^{k_t} u \in W_h$ defined by
 \begin{equation}
  (\Pi^{k_t} u )|_{\tilde \Omega \times I_n} := \Pi^{k_t,n} (u|_{\tilde \Omega \times I_n}) \quad \forall n
 \end{equation}
\end{definition}

Next, we introduce a further operator, $\Pi_{Osw}$, which maps spatially discontinuous Galerkin discrete functions on $\mathcal{E}(Q^{\text{lin},n})$ to corresponding continuous Galerkin functions. It is called Oswald operator in the literature and relies on the idea of averaging out all contributions from degrees of freedom associated to more than one (neighbouring) elements with equal weights. In the notation introduced above, it will map spatial functions from $ V_{DG}^{k_s}$ to $ V_{h}^{k_s}$, or the respective temporal tensor products of these spaces.
It holds
\begin{lemma}
 For each $I_n$, there exists an interpolation operator $\Pi_{Osw}^n$, so that for each $u \in L^2(I_n; V_{DG}^{k_s}(\mathcal{E}(\Omega^{\text{lin},n})))$ we obtain a $\Pi_{Osw}^n (u) \in L^2(I_n; V_h^{k_s}(\mathcal{E}(\Omega^{\text{lin},n})))$ such that
 \begin{equation}
  \| u - \Pi_{Osw}^n u \|_{\mathcal{E}(Q^{\text{lin},n})}^2 \lesssim \int_{I_n} \sum_{F \in \mathcal{F}_n} h_F \tripplejump{u}_F^2
 \end{equation}
\end{lemma}
A proof of this property was given in \cite{burman2007continuous}.

Again, we define a temporally global variant
\begin{definition}
 The operator $\Pi_{Osw}$ is defined for arguments $u\in W_{\text{DG}}^{n,k}, n=1,\dots, N$ to yield a discrete function $\Pi_{Osw} u \in W_h^{n,k}, n=1,\dots, N$ defined by
 \begin{equation}
  (\Pi_{Osw} u )|_{\tilde \Omega \times I_n} := \Pi_{Osw}^n (u|_{\tilde \Omega \times I_n}) \quad \forall n
 \end{equation}
\end{definition}

We can relate the facet jump appearing in the upper bound as follows to the Ghost penalty:\footnote{Note that at this occasion, the dual character of the Ghost penalty becomes relevant, as the set of facets used in e.g. \cite{FH_PHD_2025} would not suffice.}
\begin{lemma} \label{lemma_facet_jump_gp}
 It holds for $u\in W_{\text{DG}}^{n,k}, n=1,\dots,N$
 \begin{equation}
  \sum_n \int_{I_n} \sum_{F \in \mathcal{F}_n} h_F \tripplejump{u}_F^2 \lesssim h^{-j} J(u,u)
 \end{equation}
\end{lemma}
\begin{proof}
 Fix a time slice $I_n$ and a facet $F \in \mathcal{F}_n$. Let $T_1$ and $T_2$ be given such that $F = \overline{T}_1 \cap \overline{T}_2$, and denote $u_1 := u|_{T_1}$, $u_2 := u|_{T_2}$. Then, we can argue
 \begin{equation}
  h_F \tripplejump{u}_F^2 = h_F \|u_1 - u_2\|_F^2 \leq h \| u_1 - u_2 \|_{\partial T_i}^2 \lesssim \| \mathcal{E}^p (u_1 - u_2) \|_{T_i}^2 \forall  i=1,2,
 \end{equation}
 where in the last step we applied a standard trace inequality such as \cite[Eq. (3.23)]{Burman_Hansbo_Larson_Zahedi_2025}. We can identify this term with the direct Ghost penalty jump and multiply with $h^{-j}$ to cancel out the scaling.
\end{proof}

This implies for $j= -1$ in particular
\begin{corollary}
 It holds
 \begin{equation}
  \| u - \Pi_{Osw} u \|_{\mathcal{E}(Q^h)}^2 \lesssim h J(u,u). \label{eq_Oswald_detail}
 \end{equation}
\end{corollary}

\subsection{Ghost-penalty related estimates} \label{subsect_gp_estimates} The numerical analysis relies on discrete inequalities, which in the cut case are facilitated by the Ghost penalty stabilisation \cite{burman15, Burman_Hansbo_Larson_Zahedi_2025}. As it relates to the time-dependent direct Ghost-penalty, in the works \cite{preuss18, heimann20, heimann2025discretizationerroranalysishigh} such discrete inequalities were established in relation to the scaling $j=-2$, which was applied in the case with diffusion. We present the according results for the general scaling $j$ in the following, in order to apply them to our analysis involving $j=-1$ afterwards.\footnote{We mention in passing that we switched from bounds over interior elements, as in \cite{preuss18, heimann2025discretizationerroranalysishigh}, to upper bounds containing the physical domains, matching the application of these estimates later.} The results rely on a geometric assumption on the mesh, yielding a mapping from exterior to interior elements by paths over facets with activated Ghost penalty, see e.g. \cite[Ass. 5.4]{heimann2025discretizationerroranalysishigh}, which we assume to be satisfied in this article.

\begin{lemma} \label{lemma_gp_E_IpGP}
 It holds for every $u \in W_h$,
\begin{align}
 \|u\|^2_{\mathcal{E}(Q^{h,n})} &\lesssim h^{-j} j^n_h(u,u) + \| u \|_{Q^{h,n}}^2. \label{eq1_gp_E_IpGP}
\end{align}
\end{lemma}
\begin{remark}
 As for the proof, we refer the reader to \cite{preuss18, heimann2025discretizationerroranalysishigh} and note that $j^n_h$ contains the scaling $h^j$ in our setting.
\end{remark}
\begin{lemma} \label{lemma_bound_tn_ag_temporal_vol}
For any $u \in W_h$, $n=1,\dots,N$, it holds
\begin{align}
 (u^{n-1}_+, u^{n-1}_+)_{\Omega^{h}(t_{n-1}^+)} + (u^n_-, u^n_-)_{\Omega^{h}(t_n^-)} &\lesssim \frac{1}{\Delta t} \|u\|_{\mathcal{E}(Q^{h, n})}^2 \nonumber \\
 &\lesssim \frac{1}{\Delta t} \left( h^{-j} j^n_h(u,u) + \|u\|_{Q^{h, n}}^2 \right) \label{eq_bound_tn_ag_temporal_vol}
\end{align}
\end{lemma}

\begin{lemma} \label{lemma_temporal_and_spatial_inverse_ineq_gp}
 For any $u \in W_h$, it holds
\begin{align}
 \Delta t^2 (\partial_t u, \partial_t u)_{Q^{h,n}} &\lesssim \|u\|^2_{\mathcal{E}(Q^{h,n})} \lesssim h^{-j} j^n_h(u,u) + \| u \|_{Q^{h,n}}^2, \label{eq_temporal_inverse_ineq_gp}
\end{align}
\end{lemma}

\begin{lemma}\label{gp_inv_ineq}
 For any $u\in W_h$, it holds
\begin{align}
 \Delta t^2 J(\partial_t  u, \partial_t u) &\lesssim J(u,u) \label{eq_gp_inv_ineq_dt} \\
 h^2 J( \nabla u, \nabla u) &\lesssim J(u,u) \label{eq_gp_inv_ineq_grad}
\end{align}
\end{lemma}

Moreover, we can establish the following new result
\begin{lemma}
 For any $u \in W_h$, it holds
 \begin{equation}
  h^2 (\nabla u, \nabla u)_{Q^{h,n}} \lesssim \|u\|^2_{\mathcal{E}(Q^{h,n})} \lesssim h^{-j} j^n_h(u,u) + \| u \|_{Q^{h,n}}^2, \label{eq_spatial_inverse_ineq_gp}
 \end{equation}
 More generally, let $u \in W_h$ be given and $\varphi \in C^{\infty}([0,T])$, $\varphi > 0$. Then,
 \begin{equation}
  \| \varphi u \|_{\mathcal{E}(Q^h)}^2 + h^2 \| \varphi \nabla u \|_{\mathcal{E}(Q^{h})}^2 \lesssim h^{-j} \| \varphi u \|_J^2 + \| \varphi u \|_{Q^{h}}^2. \label{eq_spatial_inverse_ineq_gp_scaled}
 \end{equation}
\end{lemma}
\begin{proof}
 The proof is very similar to the temporal derivative property \Cref{lemma_temporal_and_spatial_inverse_ineq_gp}. In a first step, we bound the norm by extending the domain to the tensor-product domain $\mathcal{E}(Q^{h,n})$. There, we can apply for any fixed time the standard discrete inequality in space. Lastly, we employ \Cref{lemma_gp_E_IpGP} for moving to the physical domain at the cost of a Ghost penalty summand. Relating to \Cref{eq_spatial_inverse_ineq_gp_scaled}, we note that all discrete arguments relate to spatial functions, so that the argument generalises for a temporal scaling as stated: The extension and inverse estimates are applied at each fixed time and integrated with the weight $\varphi^2$, which also yields the bound on $\| \varphi u \|_{\mathcal{E}(Q^h)}$.
\end{proof}
\subsection{Mass conserving bilinear form and trace inequality}
In addition, we define a variant of the bilinear form $B_h$, called $B_{mc}$, which stands for mass conserving bilinear form: 
\begin{align}
 B_{mc}(u,v) = & ~ (u, -\partial_t v - \vec{w} \cdot \nabla v)_{Q^{h}} \label{eq_intro_Bmc}
 +\sum_{n=1}^{N-1} - (u^n_-, \jump{v}^n)_{\Omega^h(t_n)} + (u^N_-, v^N_-)_{\Omega^h(t_N)}.
\end{align}
This bilinear form could be used in its own right as the computational bilinear form of consideration (cf. e.g. \cite{MYRBACK2024117245, heimann2025discretizationerroranalysishigh} for the related convection-diffusion case and Paragraph ``Mass Conserving variant method'' in \Cref{sect_num_exp} below for a numerical experiment in the transport equation case). For the purposes of this analysis, however, we mostly use it as a theoretical tool to obtain relevant positive norm contributions in the stability from symmetric testing summands.

The difference between those two formulations, which stems from the discrete geometries, can be controlled by the following Lemma.
\begin{lemma}[Mass conservation form mismatch] \label{lemma_B_Bmc_diff}
For all $u,v \in W_h$ or in $H^1(Q^{h,n})$ for each $n$, the following holds
\begin{equation}
| B_h(u,v) - B_{mc}(u,v) | \leq C_{\ref{lemma_B_Bmc_diff}} \left(h + \Delta t\right) \left( \int_0^T \int_{\partial \Omega^h(t)} |uv| \, \mathrm{d}(\partial \Omega^h(t)) \, \mathrm{d}t \right). \label{eq_B_Bmc_diff}
\end{equation}
\end{lemma}
\begin{proof}
 Fix a time slice, i.e. $n \in \{ 1,\dots, N \}$.
Applying partial integration on $Q^{h,n}$ in $d+1$ dimensions yields, denoting the space-time outward unit normal by $n_{\partial Q^h}$
 \begin{align*}
  (\partial_t u + \mathrm{div}( \vec{w} \cdot u),v)_{Q^{h,n}} = -(u, \partial_t v + \vec{w} \nabla v)_{Q^{h,n}} +
  \int_{\partial Q^{h,n}} u \begin{pmatrix} \vec{w} v \\v \end{pmatrix} \cdot n_{\partial Q^h} \ \mathrm{d}(\partial Q^{h,n})
 \end{align*}
 On the lateral boundary $\{(x,t) \, | \, x \in \partial \Omega^h(t), \, t \in I_n\}$, let $\nu_h$ denote the outward spatial unit normal of $\Omega^h(t)$ and $V_h := - \partial_t \phi^{\mathrm{lin}} / |\nabla \phi^{\mathrm{lin}}|$ its normal velocity, so that $n_{\partial Q^h} = (\nu_h, -V_h)^T / \sqrt{1 + V_h^2}$ and $\mathrm{d}(\partial Q^{h,n}) = \sqrt{1 + V_h^2} \, \mathrm{d}(\partial \Omega^h(t)) \, \mathrm{d}t$. We abbreviate $\rho_h := \vec{w} \cdot \nu_h - V_h = \sqrt{1 + V_h^2} \, (\vec{w}, 1)^T \cdot n_{\partial Q^h}$.
 Separating spatial and temporal boundaries, we obtain
 \begin{align*}
  \int_{\partial Q^{h,n}} u \begin{pmatrix} \vec{w} v \\v \end{pmatrix} \cdot n_{\partial Q^h} \ \mathrm{d}(\partial Q^{h,n}) = & ~(u^n_-,v^n_-)_{\Omega^h(t_{n}^-)} - (u^{n-1}_+,v^{n-1}_+)_{\Omega^h(t_{n-1}^+)} \label{interm_step_34} \\
  &+ \int_{t_{n-1}}^{t_n} \int_{\partial \Omega^h(t)} \rho_h \, u v \, \mathrm{d} (\partial \Omega^h(t)) \, \mathrm{d}t \nonumber
 \end{align*}
 where the first two summands exactly match the difference between the DG-in-time jump contributions in both $B_h$ and $B_{mc}$.
 
 The third summand is an error term relating to the discrete geometries, as in the exact case the domain boundary moves with the convection field, i.e. $n_{\partial Q} \cdot (\vec{w}, 1)^T = 0$ on the spatial boundary of $\partial Q$. We then complete the proof by putting forward a geometric argument to estimate $ \| n_{\partial Q^h} \cdot (\vec{w}, 1)^T \|_{\infty, \partial \Omega^h}$ based on the small distance between the two, and making use of a bound on discrete and exact normals. Note that $|V_h| \leq \| \partial_t \phi^{\mathrm{lin}} \|_\infty / \min |\nabla \phi^{\mathrm{lin}}| \lesssim 1$ for $h$ and $\Delta t$ sufficiently small by the non-degeneracy of $\phi$ and \Cref{eq_phi_diff}, so that $\| \rho_h \|_{\infty, \partial \Omega^h} \leq (1 + \| V_h \|_\infty) \| n_{\partial Q^h} \cdot (\vec{w}, 1)^T \|_{\infty, \partial \Omega^h}$.

 In technical terms, in order to translate between $\partial Q^h = \partial Q^{\text{lin}}$ and $\partial Q$, we apply the mapping $\Phi_h$ as introduced in \cite{heimann2024geometrically} for $q_s = q_t = 1$.
Thereby, we obtain
\begin{align*}
 & \left\| \begin{pmatrix} \vec{w} \\ 1\end{pmatrix} \cdot n_{\partial Q^h} \right\|_{\infty, \partial \Omega^h} \hspace{-0.1cm} \leq \left| \begin{pmatrix} \vec{w} - \vec{w} \circ \Phi_h \\ 0\end{pmatrix} \cdot n_{\partial Q^h} \right|_\infty + \left|\begin{pmatrix} \vec{w} \circ \Phi_h \\ 1 \end{pmatrix} \cdot \left( n_{\partial Q^h}- n_{\partial Q} \circ \Phi_h \right) \right|_\infty \\
 &\lesssim (h^2 + \Delta t^2) \| \vec{w} \|_{W^{1, \infty}(\partial Q^h)} + (h + \Delta t) (1 + \| \vec{w}\|_{\infty, (\partial Q^h)}) \\
 &\lesssim (h + \Delta t) (1 + \| \vec{w} \|_{W^{1, \infty}(\partial Q^h)})
\end{align*}
where we applied $\| \Phi_h - \mathrm{Id} \| \lesssim h^{2} + \Delta t^{2}$ (cf. \cite[Cor. 5.2]{heimann2024geometrically}) and $\| n_{\partial Q} \circ \Phi_h - n_{\partial Q^h} \|_{\infty, \partial Q^h} \lesssim h + \Delta t$ (cf. \cite[Lemma 5.4]{heimann2024geometrically}). Summing over $n$, we have $B_h(u,v) - B_{mc}(u,v) = \int_0^T \int_{\partial \Omega^h(t)} \rho_h \, uv \, \mathrm{d}(\partial \Omega^h(t)) \, \mathrm{d}t$, and the claim follows from $\| \rho_h \|_{\infty, \partial \Omega^h} \lesssim h + \Delta t$.
\end{proof}
\begin{remark}
 The estimate $ \| n_{\partial Q^h} \cdot (\vec{w}, 1)^T \|_{\infty, \partial \Omega^h} \lesssim h + \Delta t$ is the hinge between geometry approximation and stability of the method. It could be used as an abstract weaker sufficient condition for the discrete geometry approximation instead of $\phi \in C^2$, which we have used to establish it by the perturbation argument.
\end{remark}
For the convenient presentation of such terms, we introduce the following straightforward notation for a norm on the spatial boundary:
\begin{equation}
 \| u \|_{\partial_s Q^h}^2 := \int_0^T \int_{\partial \Omega^h(t)} u^2 \, \mathrm{d}(\partial \Omega^h(t)) \, \mathrm{d}t
\end{equation}
\subsection{Physical and discrete material space-time derivatives}
In the stability proof, we gain control over a discrete version of the material derivative. Note that for a discrete function $u \in W_h$, $\vec{w} \cdot \nabla u \notin W_h$. To circumvent this, the Oswald operator for projecting spatially DG functions into the CG space is supplemented with a discrete version of $\vec{w}$ as follows:
\begin{assumption}
 We assume there exists a spatially elementwise $P^1$ and temporally on $I_n$ constant function $\vec{w}_1 \in W_h^{n,(1,0)}, n =1,\dots,N$, approximating $\vec{w}^e$, s.t.
 \begin{equation}
  \| \vec{w}^e - \vec{w}_1 \|_{\mathcal{E}(Q^{h,n}), \infty} \lesssim h^{\frac{1}{2}}. \label{eq_w_w1_diff}
 \end{equation}
 Moreover, we assume $M_1 := \max_n \| \vec{w}_1 \|_{W^{1,\infty}(\mathcal{E}(Q^{h,n}))} \lesssim 1$ uniformly in $h$, which holds for the interpolants discussed below as $\vec{w}^e \in W^{1,\infty}(U)$.
\end{assumption}
This is justified by first observing that for $h$ and $\Delta t$ sufficiently small, on each $I_n$, $\vec{w}^e$ will be defined on the tensor product space-time domain $\mathcal{E}(Q^{h,n})$. Then, fixing some $t \in I_n$ we can choose e.g. the spatial FE interpolator such as presented in \cite[Eq. (4.4.22)]{brennerscott} with $s=0, m=1$. Hence, for $d=2$ we require $\| \nabla \vec{w}^e(\cdot,t) \|_{L^4(Q^h)} < \infty$, and for $d=3$ $\| \nabla \vec{w}^e(\cdot,t) \|_{L^6(Q^h)} < \infty$. Moreover, fix some $x \in \mathcal{E}(Q^{h,n})$ and assume that $\| \partial_t \vec{w}^e(x,\cdot) \|_\infty < \infty$. Then, picking e.g. the candidate $\vec{w}^e(t_{n-1})$ suffices as by elementary arguments, $\vec{w}^e(t) = \vec{w}^e(t_{n-1}) + \int_{t_{n-1}}^t \mathrm{d} \tau \partial_t \vec{w}^e(\cdot,\tau)$, so that $\| \vec{w}^e - \vec{w}^e(t_{n-1})\|_{I_n, \infty} \lesssim \Delta t \| \partial_t \vec{w}^e \|_\infty$.

\begin{lemma} \label{lemma_gp_inv_ineq_w1grad}
 For any $u \in W_h$ and $\vec{w}_1$ as above, it holds
 \begin{equation}
  h^2 J(\vec{w}_1 \cdot \nabla u, \vec{w}_1 \cdot \nabla u) \lesssim M_1^2 \left( J(u,u) + h \| u \|_{Q^h}^2 \right). \label{eq_gp_inv_ineq_w1grad}
 \end{equation}
\end{lemma}
\begin{proof}
 On $T_i \subset \omega_F$, we have $\jump{\vec{w}_1 \cdot \nabla u}_{\omega_F} = \vec{w}_1 \cdot \jump{\nabla u}_{\omega_F} + \jump{\vec{w}_1}_{\omega_F} \cdot \mathcal{E}^p(\nabla u|_{T_{3-i}})$. As $\vec{w}_1$ is continuous and elementwise linear, $\jump{\vec{w}_1}_{\omega_F}$ is linear and vanishes on $F$, hence $|\jump{\vec{w}_1}_{\omega_F}| \lesssim h \, M_1$ on $\omega_F$. The first summand is bounded by \cref{eq_gp_inv_ineq_grad}. By the stability of the polynomial extension on shape-regular patches, the second summand contributes $h^{2} \cdot h^{j} \cdot h^2 M_1^2 \| \nabla u \|_{\mathcal{E}(Q^h)}^2 \lesssim h M_1^2 \| u \|_{\mathcal{E}(Q^h)}^2$ for $j=-1$ by the spatial inverse inequality, and \cref{eq1_gp_E_IpGP} concludes the proof.
\end{proof}

\begin{definition}
 We define the space-time material derivative operator $D_t u$ and a discrete counterpart $D_t^h$ with $D_t^h W_h \subseteq W_h$ as
 \begin{equation}
  D_t u := \partial_t u + \vec{w} \cdot \nabla u, \quad D_t^h u := \partial_t u + \Pi_{Osw} (\vec{w}_1 \cdot \nabla u)
 \end{equation}
\end{definition}
Measured in the $L^2$ norm in space and time over $Q^h$, we obtain for the difference between the two
\begin{lemma}
It holds
\begin{equation}
 h \| D_t u - D_t^h u \|_{Q^h}^2 \lesssim \| u \|_{Q^h}^2 + J(u,u) \lesssim \tripplenorm{u}_j^2. \label{eq_bound_diff_Dt_DtTheta}
\end{equation}
\end{lemma}
\begin{proof}
 We calculate
 \begin{align}
  h \| \vec{w} \nabla u - &  \Pi_{Osw} (\vec{w}_1 \nabla u) \|^2_{Q^h} \lesssim h \| \vec{w} \nabla u - \vec{w}_1 \nabla u\|_{Q^h}^2 + h \| (\mathrm{Id} -\Pi_{Osw}) (\vec{w}_1 \nabla u) \|^2_{Q^h} \nonumber \\
  &\lesssim h^2 \| \nabla u \|_{Q^h}^2 + h^2 J(\vec{w}_1 \nabla u, \vec{w}_1 \nabla u) \text { by } \cref{eq_w_w1_diff}, \cref{eq_Oswald_detail} \nonumber \\
  &\lesssim (1 + h M_1^2) \| u \|_{Q^h}^2 + (h + M_1^2) J(u,u) \text { by } \cref{eq_spatial_inverse_ineq_gp}, \cref{eq_gp_inv_ineq_w1grad}. \nonumber
 \end{align}
 All of these terms are bounded by $\tripplenorm{u}_j^2$.
\end{proof}

\subsection{Space-Time derivatives of and temporally scaled discrete functions in $\tripplenorm{\cdot}$}
We continue with an estimate on the discrete material derivative in our problem specific norm $\tripplenorm{\cdot}_j$, as well as a similar bound on the temporally rescaled variant of $u$. Both of these results will be combined in the stability proof to yield $\tripplenorm{ \tilde v(u) }_j \lesssim \tripplenorm{u}_j$, for $\tilde v(u)$ a linear combination of the two.

\begin{lemma} \label{lemma_tripplenorm_Dtu_bnd_u}
 It holds for all $u \in W_h$,
 \begin{equation}
  \tripplenorm{h D_t^h u}_j \lesssim \tripplenorm{u}_j. \label{eq_hDtu_tripplenorm_bnd_u}
 \end{equation}
\end{lemma}
\begin{proof} We prove the property using squared norms.
 Starting with the Ghost-penalty summand on the left hand side, we observe
 \begin{align}
 & h^2 \| D_t^h u \|_J^2 \lesssim h^2 \| \partial_t u \|_J^2 + h^2 \| \Pi_{Osw} (\vec{w}_1 \nabla u) \|_J^2  \nonumber \\
  \lesssim & \underbrace{ h^2 J(\partial_t u, \partial_t u) }_{\lesssim J(u,u) \text{ by } \cref{eq_gp_inv_ineq_dt}} + \underbrace{ h^2 J(\vec{w}_1 \nabla u, \vec{w}_1 \nabla u)}_{\lesssim J(u,u) + h \|u\|_{Q^h}^2 \text{ by } \cref{eq_gp_inv_ineq_w1grad}} + h^2 \| (\mathrm{Id} - \Pi_{Osw})(\vec{w}_1 \nabla u) \|_J^2. \nonumber
 \end{align}
 As $J(u,u)$ and $\| u \|_{Q^h}^2$ are among the summands of $\tripplenorm{u}_j^2$, it remains to estimate the error of the Oswald interpolator, where we note that the integral over the patches can be bounded by the global $L^2$ norm with $h^j$ scaling:
 \begin{align*}
   h^2 \| (\mathrm{Id} - \Pi_{Osw})(\vec{w}_1 \nabla u) \|_J^2  &\lesssim h^{j+2} \| (\mathrm{Id} - \Pi_{Osw})(\vec{w}_1 \nabla u)\|_{\mathcal{E}(Q^h)}^2 \\
   &\lesssim h^2 J(\vec{w}_1 \nabla u, \vec{w}_1 \nabla u) \text { by } \cref{eq_Oswald_detail} \\
   &\lesssim M_1^2 (J(u,u) + h \|u\|_{Q^h}^2) \text{ by } \cref{eq_gp_inv_ineq_w1grad}.
 \end{align*}
 Taking the two last results together yields
 \begin{equation}
  h^2 \| D_t^h u \|_J^2 \lesssim (1 + M_1^2) (J(u,u) + h \|u\|_{Q^h}^2) \lesssim J(u,u) + h \|u\|_{Q^h}^2. \label{eq_J_Dt_u_estimate}
 \end{equation}
 Now we can consider the next summand
 \begin{align*}
  \Delta t^3 \| (\partial_t + \vec{w} \nabla)(D_t^h u)\|_{Q^h}^2 \lesssim \Delta t^3 \| \partial_t D_t^h u \|_{Q^h}^2 + \Delta t^3 \|\vec{w}\|_\infty^2 \| \nabla (D_t^h u)\|_{Q^h}^2,
 \end{align*}
 where in particular
 \begin{align*}
  \Delta t^3 \| \partial_t D_t^h u \|_{Q^h}^2 &\lesssim \Delta t ( h J(D_t^h u, D_t^h u) + \| D_t^h u\|^2_{Q^h}) \text{ by } \cref{eq_temporal_inverse_ineq_gp} \\
  &\lesssim \tripplenorm{u}_j^2 + \Delta t \| D_t u - D_t^h u \|^2_{Q^h} \lesssim \tripplenorm{u}_j^2 \text{ by } \cref{eq_bound_diff_Dt_DtTheta}, \cref{eq_J_Dt_u_estimate}.
 \end{align*}
 and for the other summand,
 \begin{align*}
  \Delta t^3 \| \nabla (D_t^h u)\|_{Q^h}^2&\lesssim \Delta t ( h J(D_t^h u, D_t^h u) + \| D_t^h u\|^2_{Q^h}) \text{ by } \cref{eq_spatial_inverse_ineq_gp} \\
  &\lesssim \tripplenorm{u}_j^2 \text{ as before.}
 \end{align*}
 
 Next, we note that the $L^2$ volume summand is bounded by triangle inequality and \cref{eq_bound_diff_Dt_DtTheta} as
 \begin{align}
  h^2 \| D_t^h u \|_{Q^h}^2 \lesssim h \| D_t^h u \|_{Q^h}^2 \lesssim h \| D_t u \|_{Q^h}^2 + h \| D_t u - D_t^h u \|_{Q^h}^2 \lesssim \tripplenorm{u}_j^2. \label{eq_bnd_vol_Dthu}
 \end{align}

 It remains to estimate $\tripplejump{h D_t^h u}^2$. We observe
 \begin{align}
  \tripplejump{h D_t^h u}^2 &\lesssim h^2 \left( \sum_{n=0}^{N-1} \| (D_t^h u)_+^n\|_{\Omega^h(t_n)}^2 + \sum_{n=1}^{N} \| (D_t^h u)_-^n\|_{\Omega^h(t_n)}^2 \right) \label{eq_jump_Dt_u_Gp_est} \\
  &\lesssim h \left( \sum_{n=1}^{N} h j_h^n(D_t^h u, D_t^h u) + \| D_t^h u \|_{Q^{h,n}}^2 \right) \text{ by } \cref{eq_bound_tn_ag_temporal_vol} \nonumber \\
  &\lesssim h^2 J(D_t^h u, D_t^h u) + h \| D_t^h u\|^2_{Q^h} \lesssim \tripplenorm{u}_j^2, \text{ as in } \cref{eq_J_Dt_u_estimate}, \cref{eq_bnd_vol_Dthu}. \nonumber
 \end{align}

\end{proof}
A similar boundedness property holds for the temporally rescaled discrete version of $u$, which is a discrete projection of $u e^{-\alpha t}$ for some non-negative scalar $\alpha$ to be defined in detail below.
\begin{lemma}
 It holds for all $ u \in W_h$
 \begin{equation}
  \tripplenorm{\Pi^{k_t} (u e^{-\alpha t})}_j \lesssim \tripplenorm{u}_j. \label{eq_Pikt_ue_tripple_norm_bnd_tripple_norm_u}
 \end{equation}
\end{lemma}
\begin{proof}We prove the property using squared norms.
 Again, we argue for the summands individually. First, consider
 \begin{align}
  \| \Pi^{k_t} (u e^{-\alpha t})\|_J^2 &\lesssim \| u e^{-\alpha t} \|_J^2 + \|  (\mathrm{Id} - \Pi^{k_t}) (u e^{-\alpha t})\|_J^2 \nonumber \\
  &\lesssim \| u \|_J^2 + h^j \| (\mathrm{Id} - \Pi^{k_t}) (u e^{-\alpha t}) \|_{\mathcal{E}(Q^h)}^2 \nonumber \\
  &\lesssim \| u \|_J^2 + C(\alpha) \cdot h \cdot \| u \|_{\mathcal{E}(Q^h)}^2 \text{ by } \cref{eq_superinterpol} \nonumber \\
  &\lesssim \| u \|_J^2 + \| u \|_{Q^h}^2 \lesssim \tripplenorm{u}_j^2 \label{eq_J_norm_of_Pikt_ue} \text{ by } \cref{eq1_gp_E_IpGP}.
 \end{align}
 In relation to the volume term, by similar arguments,
 \begin{align*}
  \| \Pi^{k_t} (u e^{-\alpha t}) \|_{Q^h}^2 &\lesssim \| u e^{-\alpha t} \|_{Q^h}^2 + \| (\mathrm{Id} - \Pi^{k_t}) (u e^{-\alpha t}) \|_{Q^h}^2 \\
  &\lesssim \| u \|_{Q^h}^2 + C(\alpha) \cdot h^2 \cdot \| u\|_{\mathcal{E}(Q^h)}^2 \lesssim \tripplenorm{u}_j^2 \text{ by } \cref{eq_superinterpol}, \cref{eq1_gp_E_IpGP}.
 \end{align*}

 Next, we can move from time-slice boundaries to the volume by means of a trace inequality on the boundary of the time interval $I_n$:
 \begin{align}
  &\tripplejump{\Pi^{k_t} (u e^{-\alpha t})}^2 \lesssim \tripplejump{u e^{-\alpha t}}^2 + \tripplejump{ (\mathrm{Id} - \Pi^{k_t})(u e^{-\alpha t})}^2 \nonumber \\
  &\lesssim \tripplejump{u}^2 + \frac{1}{\Delta t} \| (\mathrm{Id} - \Pi^{k_t}) (u e^{-\alpha t})\|_{\mathcal{E}(Q^h)}^2 + \Delta t \| \partial_t (\mathrm{Id} - \Pi^{k_t}) (u e^{-\alpha t})\|_{\mathcal{E}(Q^h)}^2 \nonumber \\
  &\lesssim \tripplejump{u}^2 + C(\alpha) \cdot h \cdot \| u\|^2_{\mathcal{E}(Q^h)} + C(\alpha) \cdot h^3 \cdot (\| \partial_t u\|_{\mathcal{E}(Q^h)}^2 + \| u\|_{\mathcal{E}(Q^h)}^2) \text{ by } \cref{eq_superinterpol}, \cref{eq_superinterpol_dt} \nonumber \\
  &\lesssim \tripplenorm{u}_j^2 + C(\alpha) \cdot h \cdot ( \| u \|_{Q^h}^2 + h J(u,u)) \lesssim \tripplenorm{u}_j^2 \text{ by } \cref{eq1_gp_E_IpGP}, \cref{eq_temporal_inverse_ineq_gp}. \label{bound_fixed_time_level_Pikt_Id}
 \end{align}
 
 Finally, starting from triangle inequality
 \begin{align}
  \Delta t \| (\partial_t + \vec{w} \nabla) \Pi^{k_t} ( u e^{-\alpha t} )\|^2_{Q^h} \lesssim & ~ \Delta t \| (\partial_t + \vec{w} \nabla) (u e^{-\alpha t}) \|^2_{Q^h} \nonumber \\
  &+ \Delta t \| (\partial_t + \vec{w} \nabla) \left[(\mathrm{Id} - \Pi^{k_t}) (u e^{-\alpha t}) \right] \|^2_{Q^h}. \nonumber
 \end{align}
 Regarding the first term, we apply product rule yielding
 \begin{equation*}
  \Delta t \| (\partial_t + \vec{w} \nabla) (u e^{-\alpha t}) \|^2_{Q^h} \lesssim \Delta t \| e^{-\alpha t} (\partial_t + \vec{w} \nabla) u \|^2_{Q^h} + \Delta t \alpha^2 \| e^{-\alpha t} u \|^2_{Q^h} \lesssim \tripplenorm{u}^2.
 \end{equation*}
 Regarding the second term, we apply the following respective bounds
 \begin{align*}
  \Delta t \| (\partial_t + \vec{w} \nabla) \left[(\mathrm{Id} - \Pi^{k_t}) (u e^{-\alpha t}) \right] \|^2_{Q^h} \lesssim & \, \Delta t^3 C(\alpha) (\| u \|^2_{\mathcal{E}(Q^h)} + \| \partial_t u \|^2_{\mathcal{E}(Q^h)}) \textnormal{ by } \cref{eq_superinterpol_dt}\\
  & + \Delta t |\vec{w}|_\infty^2 \| \nabla \left[(\mathrm{Id} - \Pi^{k_t}) (u e^{-\alpha t}) \right] \|^2_{Q^h}\\
  \lesssim \Delta t C(\alpha) ( \| u \|_{Q^h}^2 + h J(u,u)) + \Delta t |\vec{w}|_\infty^2 &\| (\mathrm{Id} - \Pi^{k_t}) (e^{-\alpha t} \nabla u) \|^2_{Q^h} \textnormal{ by } \cref{eq_temporal_inverse_ineq_gp}, \cref{eq_Pikt_nabla_commute}\\
  \lesssim \tripplenorm{u}_j^2 + \Delta t^3 |\vec{w}|_\infty^2 C(\alpha) \| \nabla u \|_{\mathcal{E}(Q^h)}^2 &\lesssim \tripplenorm{u}_j^2  \textnormal{ by } \cref{eq_superinterpol}, \cref{eq_spatial_inverse_ineq_gp}.
 \end{align*}
\end{proof}

\subsection{Symmetric testing}
 In this subsection, we aim at obtaining control over the symmetric testing contribution in bilinear form and Ghost penalty. We proceed in steps about first testing with $u e^{- \alpha t}$, and then its discrete variant. Before that, we establish a special trace inequality, which is needed to control error terms stemming from the application of \Cref{lemma_B_Bmc_diff}.
 \begin{lemma}[Special trace inequality]
  For all $u \in W_h$ or in $H^1(\mathcal{E}(Q^{h,n}))$ for all $n$, it holds
  \begin{equation}
    \| u \|^2_{\partial_s Q^h} \lesssim h^{-1} \| u \|_{\mathcal{E}(Q^h)}^2 + h \| \nabla u \|_{\mathcal{E}(Q^h)}^2. \label{eq_special_trace_nondiscrete}
  \end{equation}
  Specially for $u \in W_h$, we have
  \begin{equation}
    \| u \|^2_{\partial_s Q^h} \leq \frac{C_{\ref{eq_special_trace}}}{h} ( \| u \|_{Q^h}^2 + h J(u,u)) \label{eq_special_trace}.
  \end{equation}
 \end{lemma}
 \begin{proof}
  We subdivide all parts of the boundary $\partial \Omega^h(t)$ for some $t \in I_n$ onto the respective element $T \in \mathcal{T}_h$ with $T \times I_n \subseteq \mathcal{E}(Q^h)$. There, the following cut trace inequality holds (cf. for instance \cite[Eq. (3.24)]{Burman_Hansbo_Larson_Zahedi_2025})
  \begin{equation}
   \| u (\cdot,t) \|_{\partial \Omega^h(t) \cap T}^2 \lesssim h^{-1} \| u (\cdot,t) \|_T^2 + h \| \nabla u (\cdot,t) \|_T^2.
  \end{equation}
  Integrating over time yields the estimate. For discrete functions $u \in W_h$, we argue further as follows
  \begin{align}
   h^{-1} \| u \|_{\mathcal{E}(Q^h)}^2 + h \| \nabla u \|_{\mathcal{E}(Q^h)}^2 \lesssim h^{-1} \| u\|_{Q^h}^2 + J(u,u) \text{ by } \cref{eq1_gp_E_IpGP}, \cref{eq_spatial_inverse_ineq_gp}. \nonumber
  \end{align}
 \end{proof}
\begin{lemma} \label{lemma_B_symmetric_testing_contrib}
 There exists a constant $\tilde C_{\ref{eq_special_trace}}$, depending on the constant of the special trace inequality and related discrete inequalities, on the mismatch constant $C_{\ref{lemma_B_Bmc_diff}}$ of \Cref{lemma_B_Bmc_diff} and on $c_t$ from Assumption \ref{assume_h_delta_t_equiv}, but independent of mesh size $h$, time step $\Delta t$ and the discrete geometry, so that it holds that for all $u \in W_h$
 \begin{equation}
B_h(u, u e^{-\alpha t}) \geq C_{\ref{lemma_B_symmetric_testing_contrib}} e^{-\alpha T} \left( \tripplejump{u}^2 + (\alpha - |\mathrm{div}(\vec{w})|_\infty - \tilde C_{\ref{eq_special_trace}}) \| u\|_{Q^h}^2 \right) -  C h J(u,u), \label{eq_B_symmetric_testing_contrib}
 \end{equation}
 where $\alpha > |\mathrm{div}(\vec{w})|_\infty + \tilde C_{\ref{eq_special_trace}}$.
\end{lemma}
\begin{proof}
We start with the insight from \Cref{lemma_B_Bmc_diff} and split $B_h$ between $B_h$ and $B_{mc}$ as follows
  \begin{align*}
  2 B_h(u,u e^{-\alpha t}) \geq & B_h(u,u e^{-\alpha t}) + B_{mc}(u,u e^{-\alpha t}) - (1 + c_t) C_{\ref{lemma_B_Bmc_diff}} \cdot h \cdot \| u e^{-\alpha t/2}\|_{\partial_s Q^h}^2,
 \end{align*}
 where we used $h + \Delta t \leq (1 + c_t) h$ from Assumption \ref{assume_h_delta_t_equiv}. Observing a symmetry between the inner terms,
 \begin{align*}
  & B_h(u,u e^{-\alpha t}) + B_{mc}(u,u e^{-\alpha t}) \\
  = &\sum_{n=1}^N (\partial_t u + \mathrm{div} (\vec{w} u), e^{-\alpha t} u)_{Q^{h,n}} - (u, \partial_t(u e^{-\alpha t}) + e^{-\alpha t} \vec{w} \cdot \nabla u )_{Q^{h,n}} \\
  & + \sum_{n=1}^{N-1} (\jump{u}^n, (u^n_+ - u^n_-) e^{-\alpha t_n})_{\Omega^h(t_n)} + (u^0_+, u^0_+)_{\Omega^h(t_0)} + e^{-\alpha t_N}(u^N_-, u^N_-)_{\Omega^h(t_N)} \\
  = &\sum_{n=1}^N ( e^{-\alpha t} u, \mathrm{div} (\vec{w} u) - \vec{w} \cdot \nabla u)_{Q^{h,n}} + \alpha (e^{-\alpha t} u, u)_{Q^{h,n}} + \tripplejump{u e^{-\alpha t/2}}^2 \\
  = &\sum_{n=1}^N ( e^{-\alpha t} u, (\alpha + \mathrm{div}(\vec{w})) u)_{Q^{h,n}} + \tripplejump{u e^{-\alpha t/2}}^2.
 \end{align*}
 For the remaining boundary term, we argue as follows
 \begin{align*}
  h \cdot \| u e^{-\alpha t/2}\|_{\partial_s Q^h}^2 &\lesssim \| u e^{-\alpha t/2} \|_{\mathcal{E}(Q^h)}^2 + h^2 \| e^{-\alpha t/2} \nabla u\|_{\mathcal{E}(Q^h)}^2 \textnormal{ by } \cref{eq_special_trace_nondiscrete} \\
  & \lesssim \| u e^{-\alpha t/2} \|_{Q^h}^2 + h J(u,u) \textnormal{ by } \cref{eq_spatial_inverse_ineq_gp_scaled},
 \end{align*}
 where we used $\| e^{-\alpha t/2} u \|_J^2 \leq J(u,u)$ as $e^{-\alpha t/2} \leq 1$. We can identify the $L^2(Q^h)$ volume term as well as $C h J(u,u)$.
\end{proof}
 \begin{definition}
  We choose $\alpha$ as follows
  \begin{equation}
   \alpha := |\mathrm{div}(\vec{w})|_\infty + \tilde C_{\ref{eq_special_trace}}   + 2.
  \end{equation}

 \end{definition}
\begin{lemma} \label{lemma_B_symmetric_testing_contrib2}
It holds for $h$ sufficiently small that for all $u \in W_h$
\begin{align}
B_h(u, \Pi^{k_t} ( u e^{-\alpha t}) ) \geq \ & C_{\ref{lemma_B_symmetric_testing_contrib2}} e^{-\alpha T} \left( \tripplejump{u}^2 + \| u\|_{Q^h}^2 \right) \nonumber \\
&- C_{\ref{lemma_B_symmetric_testing_contrib2}B} \sqrt{h} (h \| D_t^h u \|_{Q^h}^2+ J(u,u)). \label{eq_B_symmetric_testing_contrib2}
 \end{align}
\end{lemma}
\begin{proof}
 We follow up from the previous result and split the difference apart:
 \begin{align*}
  B_h(u,\Pi^{k_t} ( u e^{-\alpha t})) = &B_h(u,u e^{-\alpha t}) + B_h(u, ( \Pi^{k_t}- \mathrm{Id} ) (u e^{-\alpha t})).
 \end{align*}
 We investigate in detail
 \begin{align*}
  B_h(u, ( \Pi^{k_t}- \mathrm{Id} ) (u e^{-\alpha t})) = & \ ( \partial_t u + \mathrm{div}(\vec{w} u), (\Pi^{k_t}- \mathrm{Id}) (u e^{-\alpha t}))_{Q^h} \\
  & + \sum_{n=1}^{N-1} (\jump{u}^n, ((\Pi^{k_t}- \mathrm{Id}) (u e^{-\alpha t}))^n_+)_{\Omega^{h}(t_n)}\\
  & + (u^0_+, ((\Pi^{k_t}- \mathrm{Id}) (u e^{-\alpha t}))^0_+)_{\Omega^{h}(t_0)}.
 \end{align*}
 For the volume terms, we argue as follows with Cauchy-Schwarz and Young's inequality
 \begin{align*}
   & |( \partial_t u + \mathrm{div}(\vec{w} u), ( \Pi^{k_t}- \mathrm{Id} ) (u e^{-\alpha t}))_{Q^h}| \leq \|\partial_t u + \mathrm{div}(\vec{w} u) \|_{Q^h} \|( \Pi^{k_t}- \mathrm{Id} ) (u e^{-\alpha t}) \|_{Q^h} \\
   &\leq { \frac{h^{3/2}}{2}} \|\partial_t u + \mathrm{div}(\vec{w} u) \|_{Q^h}^2 + {\frac{h^{-3/2}}{2}} \|( \Pi^{k_t}- \mathrm{Id} ) (u e^{-\alpha t}) \|_{Q^h}^2 \\
   &\leq h^{3/2} \| D_t u \|_{Q^h}^2 + h^{3/2} |\mathrm{div}(\vec{w})|_\infty^2 \| u \|_{Q^h}^2 + \frac{h^{-3/2}}{2} \|( \Pi^{k_t}- \mathrm{Id} ) (u e^{-\alpha t}) \|_{\mathcal{E}(Q^h)}^2 \\
   &\leq h^{3/2} \|D_t u \|_{Q^h}^2 + \left[ h^{3/2} |\mathrm{div}(\vec{w})|_\infty^2 + \frac{h^{1/2}}{2} C(\alpha)\right] \| u \|_{\mathcal{E}(Q^h)}^2 \text{ by } \cref{eq_superinterpol} \\
   &\lesssim h^{3/2} \|D_t^h u \|_{Q^h}^2 + C \sqrt{h} J(u,u) + \left[ h^{3/2} |\mathrm{div}(\vec{w})|_\infty^2 + h^{1/2} C^\ast(\alpha)\right] \| u \|_{\mathcal{E}(Q^h)}^2 \text{ by } \cref{eq_bound_diff_Dt_DtTheta} \\
   &\leq C_D h^{3/2} \|D_t^h u \|_{Q^h}^2 + C_J \sqrt{h} J(u,u) + \frac{C_{\ref{lemma_B_symmetric_testing_contrib}}}{4} e^{-\alpha T} \| u \|_{Q^h}^2 \text{ for } h \textnormal{ sufficiently small},
 \end{align*}
 with constants $C_D$ and $C_J$ independent of $h$ and $\Delta t$, where the threshold on $h$ depends on $\alpha$ and $T$ through the factor $e^{-\alpha T}$. The temporal interface terms yield a similar result
 \begin{align*}
  &\left| \sum_{n=1}^{N-1} (\jump{u}^n, ((\Pi^{k_t}- \mathrm{Id}) (u e^{-\alpha t}))^n_+)_{\Omega^{h}(t_n)} + (u^0_+, ((\Pi^{k_t}- \mathrm{Id}) (u e^{-\alpha t}))^0_+)_{\Omega^{h}(t_0)} \right| \\
  &\leq \frac{h^{1/2}}{2} \tripplejump{u}^2 + \frac{h^{-1/2}}{2} \sum_{n=0}^{N-1} \| ((\Pi^{k_t}- \mathrm{Id}) (u e^{-\alpha t}))^n_+ \|_{\Omega^{h}(t_n)}^2 \\
  &\leq \frac{h^{1/2}}{2} \tripplejump{u}^2 + \frac{h^{1/2}}{2} C(\alpha) \| u \|_{\mathcal{E}(Q^h)}^2 \text{ by } \cref{bound_fixed_time_level_Pikt_Id}.
 \end{align*}
\end{proof}

\begin{lemma}
 For all $u \in W_h$, and $h$ sufficiently small,
 \begin{equation}
  J(u, \Pi^{k_t} (e^{-\alpha t} u)) \geq \frac{14}{16} e^{-\alpha T} J(u,u) - C h \| u \|_{Q^h}^2.
 \end{equation}
\end{lemma}
\begin{proof}
 We start with the observation that
 \begin{equation}
  J(u, e^{-\alpha t} u) \geq e^{-\alpha T} J(u,u),
 \end{equation}
 as the weighting $e^{-\alpha t}$ will just operate as a temporal scaling inside an integral over time of spatial Ghost penalty norm terms. It is always greater than or equal to $e^{- \alpha T}$, hence the estimate holds.
 
 For the remaining difference, we estimate
 \begin{align*}
  J(u, (Id - \Pi^{k_t}) (e^{-\alpha t} u)) &\leq \frac{1}{16} e^{-\alpha T} J(u,u) + 4 e^{\alpha T} \| (Id - \Pi^{k_t}) (e^{-\alpha t} u)\|_J^2 \\
  &\leq \frac{1}{16} e^{-\alpha T} J(u,u) + C e^{\alpha T} h^j \| (Id - \Pi^{k_t}) (e^{-\alpha t} u)\|_{\mathcal{E}(Q^h)}^2 \\
  &\leq \frac{1}{16} e^{-\alpha T} J(u,u) + e^{\alpha T} C(\alpha) h \| u\|_{\mathcal{E}(Q^h)}^2 \text{ by } \cref{eq_superinterpol} \\
  &\leq \frac{2}{16} e^{-\alpha T} J(u,u) + e^{\alpha T} C(\alpha) h \| u\|_{Q^h}^2,
 \end{align*}
 for $h$ sufficiently small.
\end{proof}

\begin{corollary} \label{cor_symm_testing}
 It holds $\forall u \in W_h$, and $h$ and $\Delta t$ sufficiently small
 \begin{align}
  B_h(u, \Pi^{k_t} ( u e^{-\alpha t}) ) + J(u, \Pi^{k_t} (e^{-\alpha t} u)) \geq \ & C_{\ref{cor_symm_testing}} e^{-\alpha T} \left( \tripplejump{u}^2 + \| u\|_{Q^h}^2 + J(u,u) \right) \nonumber \\
&- C_{\ref{cor_symm_testing}B} h^{3/2} \| D_t^h u\|_{Q^h}^2. \label{eq_mixed_testing}
 \end{align}
\end{corollary}

\subsection{Testing with $D_t u$}
\begin{lemma} \label{testing_u_Dtu_lemma}
 It holds for all $u \in W_h$
 \begin{equation}
  B_h(u, h D_t^h u) + J(u, h D_t^h u) \geq C_{\ref{testing_u_Dtu_lemma}A} h \| D_t^h u \|_{Q^h}^2 - C_{\ref{testing_u_Dtu_lemma}B} \left( \tripplejump{u}^2 + \| u \|^2_{Q^h} + J(u,u) \right). \label{eq_testing_u_Dtu}
 \end{equation}
\end{lemma}
\begin{proof}
 The Ghost-penalty summand is first estimated by Cauchy-Schwarz and the insight from \Cref{eq_J_Dt_u_estimate}:
 \begin{equation}
  J(u, h D_t^h u) \geq - \sqrt{J(u,u)} \sqrt{J(h D_t^h u,h D_t^h u)} \gtrsim - J(u,u) {} - h \| u \|_{Q^h}^2.
 \end{equation}
 For the $B$ summand, we observe
 \begin{align*}
  B_h(u, h D_t^h u) = & h (\partial_t u + \vec{w} \nabla u +\mathrm{div}(\vec{w}) u, D_t^h u)_{Q^h} \\& + h \sum_{n=1}^{N-1} (\jump{u}^n, (D_t^h u)^n_+)_{\Omega^{h}(t_n)} + h (u^0_+, (D_t^h u)^0_+)_{\Omega^{h}(t_0)}.
 \end{align*}
 Looking at $h (\partial_t u + \vec{w} \nabla u +\mathrm{div}(\vec{w}) u, D_t^h u)_{Q^h}$, we focus on deviations in $D_t u$ and $D_t^h u$:
 \begin{align*}
  h (\partial_t u + \vec{w} \nabla u &+\mathrm{div}(\vec{w}) u, D_t^h u)_{Q^h} = h \| D_t^h u \|_{Q^h}^2 + h ( D_t u - D_t^h u + \mathrm{div}(\vec{w}) u, D_t^h u)_{Q^h} \\
  &\gtrsim h \| D_t^h u \|_{Q^h}^2 - h \left( \frac{1}{2 \epsilon} \| D_t u - D_t^h u + \mathrm{div}(\vec{w}) u\|_{Q^h}^2+ \frac{\epsilon}{2}  \| D_t^h u\|_{Q^h}^2 \right) \\
  &\gtrsim h \| D_t^h u \|_{Q^h}^2 - h \| D_t u - D_t^h u \|_{Q^h}^2 - h \| \mathrm{div}(\vec{w}) u\|_{Q^h}^2 \\
  &\gtrsim h \| D_t^h u \|_{Q^h}^2 - (1 + h |\mathrm{div}(\vec{w})|_\infty^2) \| u \|_{Q^h}^2 - J(u,u) \text{ by } \cref{eq_bound_diff_Dt_DtTheta}.
 \end{align*}
 
 For the boundary terms, we observe by Cauchy-Schwarz
 \begin{align*}
  &\left| h \sum_{n=1}^{N-1} (\jump{u}^n, (D_t^h u)^n_+)_{\Omega^{h}(t_n)} + h (u^0_+, (D_t^h u)^0_+)_{\Omega^{h}(t_0)} \right| \\
  &\leq \tripplejump{u} \cdot \left( \sum_{n=0}^{N-1} h^2 \|(D_t^h u)^n_+\|_{\Omega^{h}(t_n)}^2 \right)^{1/2} \\
  &\leq \frac{1}{2 \epsilon_2} \tripplejump{u}^2 + \frac{\epsilon_2 }{2} \left( \sum_{n=0}^{N-1} h^2 \|(D_t^h u)^n_+\|_{\Omega^{h}(t_n)}^2 \right) \\
  &\lesssim \frac{1}{\epsilon_2} \tripplejump{u}^2 + \epsilon_2h^2 J(D_t^h u, D_t^h u) + \epsilon_2h \|D_t^h u\|^2_{Q^h} \text{ by } \cref{eq_jump_Dt_u_Gp_est} \\
&\lesssim \frac{1}{\epsilon_2} \tripplejump{u}^2 + \epsilon_2 (J(u, u) + h \| u \|_{Q^h}^2) + \epsilon_2 h \|D_t^h u\|^2_{Q^h} \text{ by } \cref{eq_J_Dt_u_estimate}.
 \end{align*}
 We can choose $\epsilon_2$ so that any such negative contributions in $h \| D_t^h u \|_{Q^h}^2$ are sufficiently small as measured against the related positive summand. This establishes the claim.
\end{proof}

\subsection{Stability}
\begin{theorem} \label{thm_stability}
 For $h \simeq \Delta t$ sufficiently small, the bilinear form $B_h + J$ is inf-sup stable on $W_h$, i.e.\ there is a constant $c > 0$ independent of $h$ and $\Delta t$ such that for all $u \in W_h \setminus \{0\}$, there exists a $\tilde v(u) \in W_h \setminus \{0\}$ with
 \begin{equation}
  B_h(u, \tilde v(u)) + J(u, \tilde v(u)) \geq c \, e^{-\alpha T} \cdot \tripplenorm{u}_j \cdot \tripplenorm{\tilde v(u)}_j.
 \end{equation}
 Hence, the problem \cref{eq_discrete_problem} is well-posed, i.e. it has a unique discrete solution.
\end{theorem}
\begin{proof}
 Let $u \in W_h$ be given. We decide to consider $\tilde v(u) = C \cdot \Pi^{k_t}( u e^{-\alpha t}) + e^{-\alpha T} h D_t^h u$, where $C = \frac{C_{\ref{testing_u_Dtu_lemma}A} + C_{\ref{testing_u_Dtu_lemma}B}}{C_{\ref{cor_symm_testing}} }$.
 
 First, we can obtain from \cref{eq_hDtu_tripplenorm_bnd_u} and \cref{eq_Pikt_ue_tripple_norm_bnd_tripple_norm_u} that $\tripplenorm{ \tilde v(u)}_j \lesssim \tripplenorm{u}_j$. Hence, it suffices to show
 \begin{align*}
  B_h(u, \tilde v(u)) + J(u, \tilde v(u)) \gtrsim e^{-\alpha T} \cdot \tripplenorm{u}_j^2.
 \end{align*}
 To establish this, we start by linearity
 \begin{align*}
  & B_h(u, \tilde v(u)) + J(u, \tilde v(u)) \\
  = \ & C \left[ B_h(u, \Pi^{k_t}( u e^{-\alpha t})) + J(u, \Pi^{k_t}( u e^{-\alpha t}))\right] + e^{-\alpha T} \left[ B_h(u, h D_t^h u) + J(u, h D_t^h u) \right] \\
  \geq \ & CC_{\ref{cor_symm_testing}} e^{-\alpha T} \left( \tripplejump{u}^2 + \| u\|_{Q^h}^2 + J(u,u) \right) - C C_{\ref{cor_symm_testing}B} h^{3/2} \| D_t^h u \|_{Q^h}^2 \text{ by } \cref{eq_mixed_testing}, \cref{eq_testing_u_Dtu} \\
 & + e^{-\alpha T} C_{\ref{testing_u_Dtu_lemma}A} h \| D_t^h u \|_{Q^h}^2 - e^{-\alpha T} C_{\ref{testing_u_Dtu_lemma}B} \left( \tripplejump{u}^2 +  \| u \|^2_{Q^h} + J(u,u)\right).
 \end{align*}
 With the choice of $C$ mentioned above, we obtain
 \begin{align*}
 & B_h(u, \tilde v(u)) + J(u, \tilde v(u)) \\
 \geq \ & (C_{\ref{testing_u_Dtu_lemma}A} + C_{\ref{testing_u_Dtu_lemma}B}) e^{-\alpha T} \left( \tripplejump{u}^2 + \| u\|_{Q^h}^2 + J(u,u) \right) - C C_{\ref{cor_symm_testing}B} h^{3/2} \| D_t^h u \|_{Q^h}^2 \\
 & + e^{-\alpha T} C_{\ref{testing_u_Dtu_lemma}A} h \| D_t^h u \|_{Q^h}^2 - e^{-\alpha T} C_{\ref{testing_u_Dtu_lemma}B} \left( \tripplejump{u}^2 + \| u \|^2_{Q^h} + J(u,u)  \right) \\
 \geq \ & C_{\ref{testing_u_Dtu_lemma}A} e^{-\alpha T} \left( \tripplejump{u}^2 + \| u\|_{Q^h}^2 + J(u,u) +  h \| D_t^h u \|_{Q^h}^2\right) - C C_{\ref{cor_symm_testing}B} h^{3/2} \| D_t^h u\|_{Q^h}^2 \\
\geq \ & c_\ast e^{- \alpha T} \tripplenorm{u}_j^2 \quad \text{for } C C_{\ref{cor_symm_testing}B} \sqrt{h} \leq \frac{1}{2} C_{\ref{testing_u_Dtu_lemma}A} e^{-\alpha T},
 \end{align*}
 where $c_\ast := C_{\ref{testing_u_Dtu_lemma}A} / (2 C_{\mathrm{eq}}) > 0$, with $C_{\mathrm{eq}}$ the constant in $\tripplenorm{u}_j^2 \leq C_{\mathrm{eq}} ( \tripplejump{u}^2 + \| u \|_{Q^h}^2 + J(u,u) + h \| D_t^h u \|_{Q^h}^2 )$, which follows from \cref{eq_bound_diff_Dt_DtTheta} by comparing the physical and the discrete material derivative. The admissible threshold on $h$ thus depends on $\alpha$, $T$ and the other fixed parameters. As the right-hand side is positive for $u \neq 0$, we have $\tilde v(u) \neq 0$, and the claim follows with $\tripplenorm{\tilde v(u)}_j \lesssim \tripplenorm{u}_j$.
\end{proof}

\subsection{Continuity}
For the purposes of showing continuity, we introduce the following larger norm
\begin{align}
  \tripplenorm{u}^2_\ast &:= \tripplenorm{u}^2 + \frac{1}{h} \| u \|_{\mathcal{E}(Q^h)}^2 + h \| \nabla u \|_{\mathcal{E}(Q^h)}^2 +\sum_{n=1}^N \| u^n_-\|_{\Omega^h(t_n)}^2
\end{align}
In terms of this norm, we can show
\begin{lemma} \label{lemma_continuity}
 For all $u \in W_h$ or in $H^1(\mathcal{E}(Q^{h,n}))$ for all $n$ and $v \in W_h$, it holds
 \begin{equation}
  B_h(u,v) \lesssim \tripplenorm{u}_\ast \tripplenorm{v}_j.
 \end{equation}
\end{lemma}
\begin{proof}
 For the proof, we follow \cite{preuss18, heimann2025discretizationerroranalysishigh} in switching to the mass conserving bilinear form, noting that
 \begin{align}
  B_h(u,v) \leq |B_{mc}(u,v)| + | B_h(u,v) - B_{mc}(u,v) |.
 \end{align}
 For the continuity bound in $B_{mc}$, we obtain $| B_{mc}(u,v) | \lesssim \tripplenorm{u}_\ast \tripplenorm{v}$ by several applications of Cauchy Schwarz inequality, noting that in particular $| (u, \partial_t v + \vec{w} \cdot \nabla v)_{Q^h} | = | (h^{-1/2} u, h^{1/2} ( \partial_t v + \vec{w} \nabla v))_{Q^h} | \lesssim \tripplenorm{u}_\ast \cdot \tripplenorm{v}$.

 For the difference $| B_h(u,v) - B_{mc}(u,v) |$, we argue as follows
 \begin{align*}
  | B_h(u,v) - B_{mc}(u,v) | &\lesssim h \cdot \| u\|_{\partial_s Q^h} \cdot \| v \|_{\partial_s Q^h} \textnormal{ by } \cref{eq_B_Bmc_diff} \\
  &\lesssim h^{1/2} \cdot \| u\|_{\partial_s Q^h} \cdot \tripplenorm{v}_j \textnormal{ by } \cref{eq_special_trace},
 \end{align*}
 where we have used that $v\in W_h$ in the last step. In relation to $u \in H^1(\mathcal{E}(Q^{h,n}))$, we argue with \Cref{eq_special_trace_nondiscrete} to obtain
 \begin{align*}
  \| u\|_{\partial_s Q^h} \lesssim h^{-1/2} \| u \|_{\mathcal{E}(Q^h)} + h^{1/2} \| \nabla u \|_{\mathcal{E}(Q^h)} \lesssim \tripplenorm{u}_\ast.
 \end{align*}
\end{proof}
\section{A priori error estimate} \label{sect_a_pr_est}
\subsection{Consistency}
\begin{lemma} \label{lemma_consistency}
 Let $u^e$ denote the extended solution to the exact problem \cref{strong_form_problem_ext}, and $v \in W_h$. Then, it holds
 \begin{equation}
  B_h(u^e,v) = f_h(v).
 \end{equation}
\end{lemma}
\begin{proof}
 Note that by virtue of Assumption \ref{u_ext_assumption}, the extended solution $u^e$ satisfies the strong form problem in a pointwise manner also on the discrete domain. Multiplying with a test function and integrating over the domain allows to establish the result, noting that $u^e$ is continuous along time slice boundaries, and satisfies the initial condition.
\end{proof}
\subsection{Interpolation and overall bound}
\begin{proposition} \label{interpol_upper_bound}
 Denoting by $u^e$ the solution to the extended continuous problem and by $u_h$ the discrete solution of \cref{eq_discrete_problem}, it holds
 \begin{equation}
  \tripplenorm{u^e - u_h} \lesssim \inf_{w_h \in W_h} \tripplenorm{ u^e - w_h}_\ast + \| w_h \|_J.
 \end{equation}
\end{proposition}
\begin{proof}
We combine \Cref{thm_stability}, \Cref{lemma_continuity}, \Cref{lemma_consistency} in the general framework of non-conforming Finite Elements, see e.g. \cite[Thm 1.35]{di2011mathematical}. (Compare also to \cite{preuss18, heimann2025discretizationerroranalysishigh, FH_PHD_2025} for a detailed derivation including $ \| w_h \|_J$.)
\end{proof}
\begin{theorem}[A priori error bound] \label{thm_a_priori_error_bnd}
 Denoting by $u^e$ the solution to the extended continuous problem \cref{strong_form_problem_ext} of regularity $\mathrm{reg}(u) \geq \max(k_s,k_t)+2$ and by $u_h$ the discrete solution of \cref{eq_discrete_problem}, it holds
 \begin{equation}
  \tripplenorm{u^e - u_h} \lesssim h^{\min(k_s,k_t) + \frac{1}{2}} \| u\|_{H^{\max(k_s,k_t)+2}(Q)}.
 \end{equation}
\end{theorem}
\begin{proof}
 Starting from \Cref{interpol_upper_bound}, we are left with the task of finding interpolation bounds. Specifically, we bound the limit by the discrete function supplied by the interpolation operator introduced in \cite{preuss18,burman_preuss_25}. Let us denote the global variant as $\Pi^{k_s,k_t}$, or $\Pi^{k}$ for $k=(k_s,k_t)$. Then, as shown in \cite[Thm A5]{burman_preuss_25}, it holds in our notation, for sufficiently smooth $u$\footnote{Note that as a choice of presentation, we will assume maximal regularity of $u$ when compared with discrete orders.}
 \begin{align}
  \| u - \Pi^{k} u \|_{\mathcal{E}(Q^h)} &\lesssim \Delta t^{k_t + 1} \| u\|_{H^{k_t+1}(\mathcal{E}(Q^h))} + h^{k_s +1} \| u\|_{H^{k_s+1}(\mathcal{E}(Q^h))}, \\
  \| \partial_t (u - \Pi^{k} u) \|_{\mathcal{E}(Q^h)} &\lesssim \Delta t^{k_t} \| u\|_{H^{k_t+1}(\mathcal{E}(Q^h))} + h^{k_s +1} \| u\|_{H^{k_s+2}(\mathcal{E}(Q^h))}, \\
  \| \nabla (u - \Pi^{k} u) \|_{\mathcal{E}(Q^h)} &\lesssim \Delta t^{k_t+1} \| u\|_{H^{k_t+2}(\mathcal{E}(Q^h))} + h^{k_s} \| u\|_{H^{k_s+1}(\mathcal{E}(Q^h))}.
 \end{align}
 We use these results to establish the bound in relation to the $L^2$ in space and time summands of $\tripplenorm{\cdot}_\ast$. The material derivative is split by a triangle inequality. Noting the scaling within $\tripplenorm{\cdot}$ yields exactly order $\min(k_s,k_t) + \frac{1}{2}$.

 Next, the norms for jump operators can be bounded by the respective fixed time norms. These terms also appear in $\tripplenorm{\cdot}_\ast$ individually and are bounded by the following estimate from \cite[Lemma A7]{burman_preuss_25} for any $t \in I_n, n=1,\dots,N$
 \begin{align*}
  \| (u - \Pi^{k} u)(t) \|_{\mathcal{E}(Q^h)|_t} \lesssim \Delta t^{k_t + \frac{1}{2}} \| u \|_{H^{k_t+1}(\mathcal{E}(Q^{h,n}))} + \Delta t^{-\frac{1}{2}} h^{k_s +1} \| u\|_{H^{k_s+1}(\mathcal{E}(Q^{h,n}))}.
 \end{align*}
 Summing over $n$ and taking into account $h \simeq \Delta t$, the summands are controlled by the same upper bound.

 In all cases, the boundedness of the extension $u^e$ is used to transfer control from the extended unfitted domain to the physical one, cf. \Cref{eq_bnd_u_e}. In detail, all Sobolev norms on the right-hand sides are bounded by the one of order $\max(k_s,k_t)+2$. As $\mathcal{E}(Q^h) \subseteq U_\delta \subseteq U$ for $h$ and $\Delta t$ sufficiently small (cf. \Cref{remark_extension_width}), we obtain $\| u^e \|_{H^{\max(k_s,k_t)+2}(\mathcal{E}(Q^h))} \leq \| u^e \|_{H^{\max(k_s,k_t)+2}(U)} \lesssim \| u \|_{H^{\max(k_s,k_t)+2}(Q)}$ by \Cref{eq_bnd_u_e}, where $u = u^e|_Q$ is the solution of \cref{strong_form_problem}.

 Finally, in relation to the Ghost penalty remainder, we can establish
 \begin{equation}
  \| \Pi^{k} u \|_J \lesssim \sqrt{\gamma_J} h^{k_s + 1/2} \| u \|_{H^{k_s+1}(\mathcal{E}(Q^h))} \label{eq_gp_consist_error}
 \end{equation}
 by the proof technique of \cite[Prop. 3.26]{preuss18}, noting that our $h$ scaling is one order smaller, $j=-1$, and $h \simeq \Delta t$ in our application.
\end{proof}
We finally point out that by virtue of the problem specific norm, the following error bound in the $L^2$ and material derivative norm follow.
\begin{corollary} \label{cor_final_bound_mat_deriv}
 Denoting by $u^e$ the solution to the extended continuous problem \cref{strong_form_problem_ext} of regularity $\mathrm{reg}(u) \geq \max(k_s,k_t)+2$ and by $u_h$ the discrete solution of \cref{eq_discrete_problem}, it holds
 \begin{align}
  \| u^e - u_h \|_{Q^h} &\lesssim  h^{\min(k_s,k_t)+ \frac{1}{2}} \| u\|_{H^{\max(k_s,k_t)+2}(Q)}, \\
  \| (\partial_t + \vec{w} \nabla) (u^e - u_h) \|_{Q^h} &\lesssim h^{\min(k_s,k_t)} \| u\|_{H^{\max(k_s,k_t)+2}(Q)}.
 \end{align}
\end{corollary}

\section{Numerical Experiments} \label{sect_num_exp}
In this section, we demonstrate the convergence property of the suggested method by numerical experiments. In particular, we prescribe a vector-field $\vec{w}$, an according moving domain, and a solution function $u$ and calculate the according right hand side $f$ based on the PDE. This allows for straightforward error calculations to assess convergence orders.

All computations are performed with the library \texttt{ngsxfem} \cite{xfem_joss}, within the general FE software \texttt{NGSolve} \cite{schoberl2014c++}. The relevant reproduction files are available at \url{https://doi.org/10.5281/zenodo.17185244}.

The technical details of the computations regarding linear solve and assembly are similar to the case with diffusion which was presented in \cite{HLP2022}: First, note that the upwind DG-in-time approach allows for sequential matrix solves within time-slices $I^n / Q^n$, thereby increasing computational feasibility. For any of these time-slice linear systems, the direct solvers \texttt{Umfpack} \cite{10.1145/992200.992206} and \texttt{pardiso} \cite{SCHENK200169} are applied. For the assembly of the matrix, the cut quadrature rules of \cite{HLP2022, fhreport} in the topology-sensitive version in time are applied, and the resulting spatial numerical integration problems on the temporal integration points follow the standard piecewise linear reference geometry approach \cite{fhreport, L_PHD_2015}. For the curved meshes, the space-time isoparametric mapping of \cite{heimann2024geometrically} is applied. The direct Ghost penalty assembly is also based on the \texttt{ngsxfem} routines developed within \cite{HLP2022, heimann2025discretizationerroranalysishigh}. Unless stated otherwise, the Ghost penalty parameter is set to $\gamma_J = 0.05$ in all experiments.
\subsection{Spatially two dimensional problem} \label{sect_num_2d} First, let us investigate the method by an example with $\mathrm{div}(\vec{w}) \neq 0$ in two spatial dimensions.
In particular, let $\tilde \Omega = [-3.5,3.5]^2$, $\beta = 1$, and $\vec{w} = \beta (x,y)^T$. The time interval of physical interest shall be $[0,T] = [0,1]$. In this time, the domain is expanding from a circle centered at the origin with radius 1 to a circle with radius $e$. This is realised by the following shifted coordinates and time dependent levelset function
\begin{equation}
 \tilde x = \exp (-\beta t) x, \quad \tilde y = \exp (-\beta t) y, \quad r = \sqrt{ \tilde x^2 + \tilde y^2}, \quad \phi = r - 1.
\end{equation}
Note that this geometry satisfies the condition of $\partial \Omega$ moving with $\vec{w}$.
We assume that the solution function $u$ becomes
$
 u = \cos (\pi r) \cdot \sin \left(\frac{\pi}{2} t \right),
$
and calculate $f$ accordingly in the manufactured solution approach.

The background domain $\tilde \Omega$ is equipped with a uniform unstructured mesh of mesh size $h$. When refining spatial and temporal resolution by means of a parameter $i=0,1, \dots$, mesh size and time step are set as
$
 h = 0.9 \cdot \left( \frac{1}{2} \right)^i, \Delta t = T \cdot \left( \frac{1}{2} \right)^{i+1}.
$
We illustrate an example mesh, discrete geometry, transport velocity, and time-dependent solution function $u$ in \Cref{num_exp_figure1}.
\begin{figure}
\begin{center}
\includegraphics[width=0.99\textwidth]{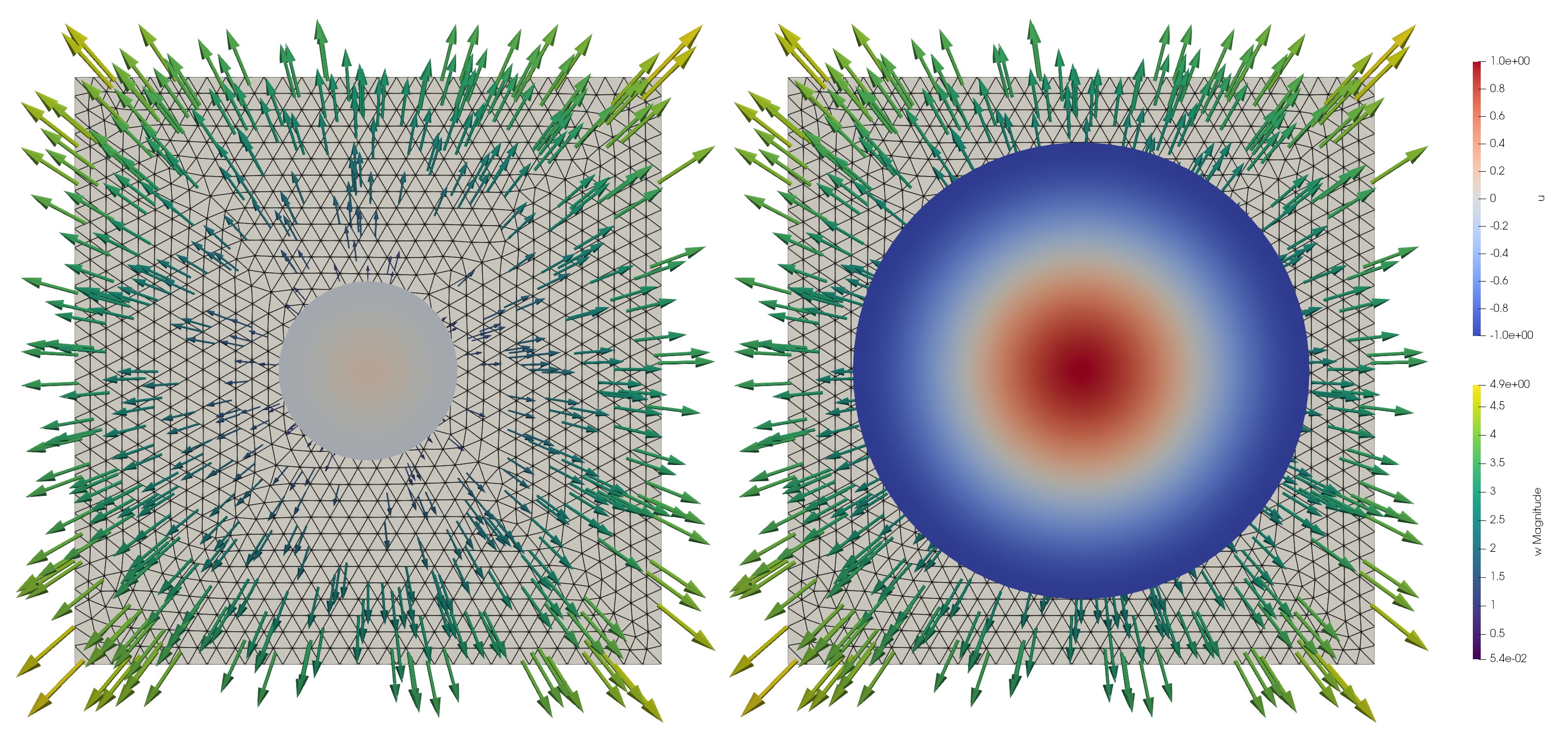}
\end{center}
\caption{Unfitted example geometry in two spatial dimensions. In the background, an unfitted finite element mesh is shown, together with the flow field $\vec{w}$ represented as arrows. In the foreground, the computational geometry with solution function $u$ is displayed. Left: At the end of the first time step, Right: At the end of the final time step.}
\label{num_exp_figure1}
\end{figure}
First, we investigate the method with uncurved meshes, which was the focus of the numerical analysis. We fix an order $k=2,\dots,6$ and set all discrete space orders to that, $k_s = k_t = k$. For the discrete geometry order in time, $q_t$, we chose $q_t=1$ as the smallest and hence computationally most favourable option. We investigate the convergence of the discrete solution function against the manufactured solution in the following norms:
\begin{align}
 \| u - u_h \|_{L^2(\Omega^h(T))}^2 &= \int_{\Omega^h(T)} (u(T) - u_h(T))^2, \\
 | u - u_h |_{H^1(Q^h)}^2 &= \sum_{n=1}^N \int_{I_n} \int_{\Omega^h(t)} (\nabla u- \nabla u_h)^2 + (\partial_t u- \partial_t u_h)^2,
\end{align}
where the latter is the $H^1$ seminorm broken across the time slabs, as $u_h$ is discontinuous in time. The results for these error norms are shown in \Cref{plot1}.
\begin{figure}
\begin{center}
   \begin{tikzpicture}[scale=0.75]
    \begin{semilogyaxis}[ xlabel=$i$, ylabel={$\| u - u_h \|_{L^2(\Omega^h(T))}$},
    legend entries ={ $k=2$, $k=3$, $k=4$, $k=5$, $k=6$}, legend style={anchor=north,legend columns=1, draw=none}, legend pos = south west,
     ]
     \addplot table [x index = 0, y index =1] {num_exp/out/conv_drift_DG_ks2_kt2_both_nref6_gamma0.05_K0.0_sm0_ktls1_ksls1_alpha0.0.dat};
     \addplot table [x index = 0, y index =1] {num_exp/out/conv_drift_DG_ks3_kt3_both_nref6_gamma0.05_K0.0_sm0_ktls1_ksls1_alpha0.0.dat};
     \addplot table [x index = 0, y index =1] {num_exp/out/conv_drift_DG_ks4_kt4_both_nref5_gamma0.05_K0.0_sm0_ktls1_ksls1_alpha0.0.dat};
     \addplot table [x index = 0, y index =1] {num_exp/out/conv_drift_DG_ks5_kt5_both_nref4_gamma0.05_K0.0_sm0_ktls1_ksls1_alpha0.0.dat};
     \addplot table [x index = 0, y index =1] {num_exp/out/conv_drift_DG_ks6_kt6_both_nref4_gamma0.05_K0.0_sm0_ktls1_ksls1_alpha0.0.dat};
     \addplot[gray, dashed, domain=0:5] {(1/2^(x+2)))^3};
     \addplot[gray, dashed, domain=0:5] {(1/2^(x+2)))^4};
     \addplot[gray, dashed, domain=0:4] {(1/2^(x+2)))^5};
     \addplot[gray, dashed, domain=0:3] {(1/2^(x+2)))^6};
     \addplot[gray, dashed, domain=0:3] {(1/2^(x+2.5)))^7};
    \end{semilogyaxis}
                  \node[scale=0.75] at (4.5,5.25) {$O(h^{k+1})= O(\Delta t^{k+1})$};
     \draw[scale=0.75, gray, dash=on 2.25pt off 2.25pt phase 0pt, line width=0.4*0.75pt] (6.25/0.75,5.25/0.75) -- (6.8/0.75,5.25/0.75);
   \end{tikzpicture}
   \begin{tikzpicture}[scale=0.75]
    \begin{semilogyaxis}[ xlabel=$i$, ylabel={$| u - u_h|_{H^1(Q^h)}$},
    legend entries ={ $k=2$, $k=3$, $k=4$, $k=5$, $k=6$}, legend style={anchor=north,legend columns=1, draw=none}, legend pos = south west,
     ]
     \addplot table [x index = 0, y index =2] {num_exp/out/conv_drift_DG_ks2_kt2_both_nref6_gamma0.05_K0.0_sm0_ktls1_ksls1_alpha0.0.dat};
     \addplot table [x index = 0, y index =2] {num_exp/out/conv_drift_DG_ks3_kt3_both_nref6_gamma0.05_K0.0_sm0_ktls1_ksls1_alpha0.0.dat};
     \addplot table [x index = 0, y index =2] {num_exp/out/conv_drift_DG_ks4_kt4_both_nref5_gamma0.05_K0.0_sm0_ktls1_ksls1_alpha0.0.dat};
     \addplot table [x index = 0, y index =2] {num_exp/out/conv_drift_DG_ks5_kt5_both_nref4_gamma0.05_K0.0_sm0_ktls1_ksls1_alpha0.0.dat};
     \addplot table [x index = 0, y index =2] {num_exp/out/conv_drift_DG_ks6_kt6_both_nref4_gamma0.05_K0.0_sm0_ktls1_ksls1_alpha0.0.dat};
     \addplot[gray, dashed, domain=0:5] {(1/2^(x+1)))^2};
     \addplot[gray, dashed, domain=0:5] {(1/2^(x+1)))^3};
     \addplot[gray, dashed, domain=0:4] {(1/2^(x+1)))^4};
     \addplot[gray, dashed, domain=0:3] {(1/2^(x+1)))^5};
     \addplot[gray, dashed, domain=0:3] {(1/2^(x+1.5)))^6};
    \end{semilogyaxis}
                  \node[scale=0.75] at (4.5,5.25) {$O(h^{k})= O(\Delta t^{k})$};
     \draw[scale=0.75, gray, dash=on 2.25pt off 2.25pt phase 0pt, line width=0.4*0.75pt] (6.25/0.75,5.25/0.75) -- (6.8/0.75,5.25/0.75);
   \end{tikzpicture}
\end{center}
\caption{Convergence of the discrete method for $k=k_s=k_t$ and uncurved meshes in the $L^2$ norm and the broken $H^1$ seminorm. }
\label{plot1}
\end{figure}
In terms of the convergence orders, we observe
\begin{equation}
 \| u - u_h \|_{L^2(\Omega^h(T))} \simeq h^{k+1} \simeq \Delta t^{k+1}, \quad | u - u_h |_{H^1(Q^h)} \simeq h^{k} \simeq \Delta t^{k}
\end{equation}
Comparing this to the a priori estimate, we note that the broken $ | \cdot |_{H^1(Q^h)}$ norm controls the combined material derivative contribution in the problem-specific norm, so that the observed order $k$ is consistent with the theoretical estimate in \Cref{cor_final_bound_mat_deriv} in relation to the material derivative norm. As it relates to the $L^2$ scaling, we observe order $\mathcal{O}(h^{k+1})$, with the specific unstructured meshes at hand, whilst \Cref{thm_a_priori_error_bnd} established $\mathcal{O}(h^{k+\frac{1}{2}})$ for the final-time $L^2$ error through the endpoint contribution $\| u^N_- \|_{\Omega^h(T)}$ of $\tripplejump{\cdot}$. This phenomenon is well-understood for convection-dominated problem stabilised continuous Galerkin methods \cite{4eb08683-2c92-3bd7-8375-4662768b17a6, 3e7704d9-d1d0-3860-8ca7-ac8179e32d8f}.

\begin{remark}[On $q_t=0$ and discrete Ghost penalty extensions] \label{remark_on_qt0}
 When defining the DG in time jump penalties, we have assumed continuity of the discrete geometry in time, to yield straightforward well-posedness of the jump operator. This motivated the choice of $q_t \geq 1$. In the light of the works on CG in \cite{HLP2022} or time stepping investigations such as \cite{LO_ESAIM_2019}, it seems interesting to combine a discontinuous in time discrete geometry stemming from $q_t = 0$ with a Ghost penalty for the purpose of discrete extension of the domain. We mention that from a numerical point of view, the discrete extension mechanism introduced for CG in \cite{HLP2022} seems to yield appropriate discrete extensions when combined with the method at hand for $q_t = 0$, but leave a further detailed investigation for future research.
\end{remark}
\paragraph{Curved Mesh variant method} Next, we give a numerical comparison to the method with curved meshes, as presented in \Cref{discrete_curved_geom_method}. The discrete errors for the same test case are displayed in \Cref{plot2}.
\begin{figure}
\begin{center}
   \begin{tikzpicture}[scale=0.75]
    \begin{semilogyaxis}[ xlabel=$i$, ylabel={$\| u - u_h \|_{L^2(\Omega^h(T))}$},
    legend entries ={ $k=2$, $k=3$, $k=4$, $k=5$, $k=6$}, legend style={anchor=north,legend columns=1, draw=none}, legend pos = south west,
     ]
     \addplot table [x index = 0, y index =1] {num_exp/out/conv_drift_DG_ks2_kt2_both_nref6_gamma0.05_K0.0_sm0_alpha0.0.dat};
     \addplot table [x index = 0, y index =1] {num_exp/out/conv_drift_DG_ks3_kt3_both_nref6_gamma0.05_K0.0_sm0_alpha0.0.dat};
     \addplot table [x index = 0, y index =1] {num_exp/out/conv_drift_DG_ks4_kt4_both_nref5_gamma0.05_K0.0_sm0_alpha0.0.dat};
     \addplot table [x index = 0, y index =1] {num_exp/out/conv_drift_DG_ks5_kt5_both_nref4_gamma0.05_K0.0_sm0_alpha0.0.dat};
     \addplot table [x index = 0, y index =1] {num_exp/out/conv_drift_DG_ks6_kt6_both_nref4_gamma0.05_K0.0_sm0_alpha0.0.dat};
     \addplot[gray, dashed, domain=0:5] {(1/2^(x+2)))^3};
     \addplot[gray, dashed, domain=0:5] {(1/2^(x+2)))^4};
     \addplot[gray, dashed, domain=0:4] {(1/2^(x+2)))^5};
     \addplot[gray, dashed, domain=0:3] {(1/2^(x+2)))^6};
     \addplot[gray, dashed, domain=0:3] {(1/2^(x+2.5)))^7};
    \end{semilogyaxis}
                  \node[scale=0.75] at (4.5,5.25) {$O(h^{k+1})= O(\Delta t^{k+1})$};
     \draw[scale=0.75, gray, dash=on 2.25pt off 2.25pt phase 0pt, line width=0.4*0.75pt] (6.25/0.75,5.25/0.75) -- (6.8/0.75,5.25/0.75);
   \end{tikzpicture}
   \begin{tikzpicture}[scale=0.75]
    \begin{semilogyaxis}[ xlabel=$i$, ylabel={$| u - u_h|_{H^1(Q^h)}$},
    legend entries ={ $k=2$, $k=3$, $k=4$, $k=5$, $k=6$}, legend style={anchor=north,legend columns=1, draw=none}, legend pos = south west,
     ]
     \addplot table [x index = 0, y index =2] {num_exp/out/conv_drift_DG_ks2_kt2_both_nref6_gamma0.05_K0.0_sm0_alpha0.0.dat};
     \addplot table [x index = 0, y index =2] {num_exp/out/conv_drift_DG_ks3_kt3_both_nref6_gamma0.05_K0.0_sm0_alpha0.0.dat};
     \addplot table [x index = 0, y index =2] {num_exp/out/conv_drift_DG_ks4_kt4_both_nref5_gamma0.05_K0.0_sm0_alpha0.0.dat};
     \addplot table [x index = 0, y index =2] {num_exp/out/conv_drift_DG_ks5_kt5_both_nref4_gamma0.05_K0.0_sm0_alpha0.0.dat};
     \addplot table [x index = 0, y index =2] {num_exp/out/conv_drift_DG_ks6_kt6_both_nref4_gamma0.05_K0.0_sm0_alpha0.0.dat};
     \addplot[gray, dashed, domain=0:5] {(1/2^(x+1)))^2};
     \addplot[gray, dashed, domain=0:5] {(1/2^(x+1)))^3};
     \addplot[gray, dashed, domain=0:4] {(1/2^(x+1)))^4};
     \addplot[gray, dashed, domain=0:3] {(1/2^(x+1)))^5};
     \addplot[gray, dashed, domain=0:3] {(1/2^(x+1.5)))^6};
    \end{semilogyaxis}
                  \node[scale=0.75] at (4.5,5.25) {$O(h^{k})= O(\Delta t^{k})$};
     \draw[scale=0.75, gray, dash=on 2.25pt off 2.25pt phase 0pt, line width=0.4*0.75pt] (6.25/0.75,5.25/0.75) -- (6.8/0.75,5.25/0.75);
   \end{tikzpicture}
\end{center}
\caption{Convergence of the discrete method for $k=k_s=k_t$ and curved meshes in the $L^2$ norm and the broken $H^1$ seminorm. }
\label{plot2}
\end{figure}
We note that the method variant shares the optimal higher order convergence property in the limit $h$ and $\Delta t \to 0$. However, in the pre-asymptotic regime, the curved meshes seem to negatively affect the absolute numerical errors. This is not uncommon in isoparametric finite elements, but might be seen as a computational argument in favour of the method with uncurved discrete geometries.

\paragraph{Mass Conserving variant method} Paralleling works in the convection-diffusion case \cite{preuss18, MYRBACK2024117245, heimann2025discretizationerroranalysishigh}, one might be interested in considering the straightforward implementation of the method involving the mass conserving bilinear form \Cref{eq_intro_Bmc}. We present the convergence results for the two-dimensional test case for comparison in \Cref{plot3}, with uncurved meshes on the left-hand side and isoparametric meshes on the right-hand side. Most clearly, the method with uncurved meshes converges only with order $\sqrt{h}$ in the $H^1$ norm, whilst the curved meshes retain higher order convergence. In the $L^2$ norm, we observe relatedly (experiments not shown in plots) $h^{3/2}$ with the uncurved meshes and full high order with curved meshes.\footnote{We mention in passing that in relation to the blending options of the isoparametric mapping as discussed in detail in \cite{heimann2024geometrically}, we have moved from the Finite Element blending used throughout this study to the smooth blending for the conservative formulation, as we have found it to yield the full orders of convergence for high order applications such as $k=4$.} We conclude that the mass conservative form seems best suited for a combination with curved computational geometries. We conjecture the observed lower order of convergence with the piecewise linear computational geometry might mathematically be traced back to an additional dominating error summand in the consistency statement, and substantiate this conjecture by the following calculation.
\begin{remark}[Consistency of the mass-conserving formulation on piecewise linear geometries]\label{remark_mc_consistency}
 We substantiate the conjecture by directly inserting $u^e$ into $B_{mc}$. Applying partial integration on each time slice $Q^{h,n}$ as in the proof of Lemma~\ref{lemma_B_Bmc_diff}, and using that $u^e$ is continuous along time-slice boundaries, and satisfies $\partial_t u^e + \mathrm{div}(\vec{w} u^e) = f$ pointwise on $Q^h$, we obtain
 \begin{equation}
  B_{mc}(u^e, v_h) - f_h(v_h) = - \int_0^T \int_{\partial \Omega^h(t)} \rho_h \, u^e \, v_h \, \mathrm{d}(\partial \Omega^h(t)) \, \mathrm{d}t,
 \end{equation}
 with $\rho_h = \vec{w} \cdot \nu_h - V_h$ the mismatch between the normal velocity of $\partial \Omega^h(t)$ and the normal component of $\vec{w}$, as in the proof of Lemma~\ref{lemma_B_Bmc_diff}. On the exact geometry, $\begin{pmatrix} \vec{w} \\ 1 \end{pmatrix} \cdot n_{\partial Q} = 0$ since $\partial\Omega(t)$ moves with $\vec{w}$, so the consistency error vanishes. On the piecewise linear discrete geometry, however, $| \rho_h | = \mathcal{O}(h + \Delta t)$ (cf.\ proof of Lemma~\ref{lemma_B_Bmc_diff}), yielding a non-trivial consistency error. On curved meshes the factor improves to high-order.
\end{remark}
\begin{figure}
\begin{center}
   \begin{tikzpicture}[scale=0.75]
    \begin{semilogyaxis}[ xlabel=$i$, ylabel={$| u - u_h|_{H^1(Q^h)}$}, title={Cons. form. $k=k_t=k_s$, $q_t=q_s=1$},
    legend entries ={ $k=2$, $k=3$, $k=4$}, legend style={anchor=north,legend columns=1, draw=none}, legend pos = north east,
     ]
     \addplot table [x index = 0, y index =2] {num_exp/out/conv_drift_DG_ks2_kt2_both_nref6_gamma0.05_K0.0_sm0_ktls1_ksls1_cons_alpha0.0.dat};
     \addplot table [x index = 0, y index =2] {num_exp/out/conv_drift_DG_ks3_kt3_both_nref6_gamma0.05_K0.0_sm0_ktls1_ksls1_cons_alpha0.0.dat};
     \addplot table [x index = 0, y index =2] {num_exp/out/conv_drift_DG_ks4_kt4_both_nref5_gamma0.05_K0.0_sm0_ktls1_ksls1_cons_alpha0.0.dat};
     \addplot[gray, dashed, domain=0:5] {(1/2^(x+1)))^0.5};
    \end{semilogyaxis}
                  \node[scale=0.75] (order_note) at (2.1,0.5) {$O(\sqrt{h})= O(\sqrt{\Delta t})$};
     \draw[scale=0.75, gray, dash=on 2.25pt off 2.25pt phase 0pt, line width=0.4*0.75pt] (order_note.west) -- ($(order_note.west) + (-0.5,0)$);
   \end{tikzpicture}
       \begin{tikzpicture}[scale=0.75]
    \begin{semilogyaxis}[ xlabel=$i$, ylabel={$| u - u_h|_{H^1(Q^h)}$}, title={Cons. form. $k=k_t=k_s = q_t=q_s$},
    legend entries ={ $k=2$, $k=3$, $k=4$}, legend style={anchor=north,legend columns=1, draw=none}, legend pos = south west,
     ]
     \addplot table [x index = 0, y index =2] {num_exp/out/conv_drift_DG_ks2_kt2_both_nref6_gamma0.05_K0.0_sm0_cons_smb_alpha0.0.dat};
     \addplot table [x index = 0, y index =2] {num_exp/out/conv_drift_DG_ks3_kt3_both_nref6_gamma0.05_K0.0_sm0_cons_smb_alpha0.0.dat};
     \addplot table [x index = 0, y index =2] {num_exp/out/conv_drift_DG_ks4_kt4_both_nref5_gamma0.05_K0.0_sm0_cons_smb_alpha0.0.dat};
     \addplot[gray, dashed, domain=0:5] {(1/2^(x+1)))^2};
     \addplot[gray, dashed, domain=0:5] {(1/2^(x+1)))^3};
     \addplot[gray, dashed, domain=0:4] {(1/2^(x+1)))^4};
    \end{semilogyaxis}
                  \node[scale=0.75] at (4.5,5.25) {$O(h^{k})= O(\Delta t^{k})$};
     \draw[scale=0.75, gray, dash=on 2.25pt off 2.25pt phase 0pt, line width=0.4*0.75pt] (6.25/0.75,5.25/0.75) -- (6.8/0.75,5.25/0.75);
   \end{tikzpicture}
\end{center}
\caption{Convergence of the discrete conservative form method for $k=k_s=k_t$ on uncurved (left hand side) and curved (right hand side) meshes in the broken $H^1$ seminorm. }
\label{plot3}
\end{figure}

\paragraph{Conservation errors} Finally, let $m_h(t) := \int_{\Omega^h(t)} u_h(t)$ denote the discrete mass, with $u_h(0) := I_h u_0$ the interpolated initial datum, and consider the conservation error $e_{\mathrm{mass}} := | m_h(T) - m_h(0) - \int_{Q^h} f |$; for the exact solution on the physical geometry, $m(T) - m(0) - \int_Q f = 0$ by the Reynolds transport theorem, as the boundary moves with $\vec{w}$. \Cref{plot_cost_mass} (left) shows $e_{\mathrm{mass}}$ for the primal method. On uncurved meshes, we observe second order convergence independent of $k$. Testing \cref{eq_discrete_problem} with $v_h = 1$ shows that $e_{\mathrm{mass}}$ equals the absolute value of the boundary term $\int_0^T \int_{\partial \Omega^h(t)} \rho_h u_h$ of \Cref{lemma_B_Bmc_diff}, cf.\ \Cref{remark_mc_consistency}, as $u_0 = I_h u_0 = 0$ in this example; the observed order $h^2$ is one order better than the bound on $\| \rho_h \|_\infty$ and consistent with the volume accuracy of the piecewise linear geometry. On curved meshes, it converges with high order down to the level of $10^{-8}$ to $10^{-10}$ attainable in double precision for this setup. The conservative formulation is mass conserving by construction, as $v_h = 1$ is an admissible test function, which is confirmed on curved meshes at the same level.

\begin{figure}
\begin{center}
   \begin{tikzpicture}[scale=0.68]
    \begin{semilogyaxis}[ xlabel=$i$, ylabel={$e_{\mathrm{mass}}$}, title={Primal method}, ymax=1e3,
    legend entries ={ $k=2$, $k=3$, $k=4$, $k=5$, $k=6$, $O(h^2)$, $O(h^4)$}, legend style={anchor=north,legend columns=2, draw=none, font=\scriptsize}, legend pos = north east,
     ]
     \addplot table [x index = 0, y index =3] {num_exp/out/conv_drift_DG_ks2_kt2_both_nref6_gamma0.05_K0.0_sm0_alpha0.0.dat};
     \addplot table [x index = 0, y index =3] {num_exp/out/conv_drift_DG_ks3_kt3_both_nref6_gamma0.05_K0.0_sm0_alpha0.0.dat};
     \addplot table [x index = 0, y index =3] {num_exp/out/conv_drift_DG_ks4_kt4_both_nref5_gamma0.05_K0.0_sm0_alpha0.0.dat};
     \addplot table [x index = 0, y index =3] {num_exp/out/conv_drift_DG_ks5_kt5_both_nref4_gamma0.05_K0.0_sm0_alpha0.0.dat};
     \addplot table [x index = 0, y index =3] {num_exp/out/conv_drift_DG_ks6_kt6_both_nref4_gamma0.05_K0.0_sm0_alpha0.0.dat};
     \addplot[Set1-A, densely dashed, mark=o, forget plot] table [x index = 0, y index =3] {num_exp/out/conv_drift_DG_ks2_kt2_both_nref6_gamma0.05_K0.0_sm0_ktls1_ksls1_alpha0.0.dat};
     \addplot[Set1-B, densely dashed, mark=o, forget plot] table [x index = 0, y index =3] {num_exp/out/conv_drift_DG_ks3_kt3_both_nref6_gamma0.05_K0.0_sm0_ktls1_ksls1_alpha0.0.dat};
     \addplot[Set1-C, densely dashed, mark=o, forget plot] table [x index = 0, y index =3] {num_exp/out/conv_drift_DG_ks4_kt4_both_nref5_gamma0.05_K0.0_sm0_ktls1_ksls1_alpha0.0.dat};
     \addplot[Set1-D, densely dashed, mark=o, forget plot] table [x index = 0, y index =3] {num_exp/out/conv_drift_DG_ks5_kt5_both_nref4_gamma0.05_K0.0_sm0_ktls1_ksls1_alpha0.0.dat};
     \addplot[Set1-E, densely dashed, mark=o, forget plot] table [x index = 0, y index =3] {num_exp/out/conv_drift_DG_ks6_kt6_both_nref4_gamma0.05_K0.0_sm0_ktls1_ksls1_alpha0.0.dat};
     \addplot[gray, dashed, line width=0.7pt, domain=1:5] {0.22*(1/2^x)^2};
     \addplot[black!70, densely dotted, line width=1pt, domain=1:5] {0.0093*(1/2^x)^4};
    \end{semilogyaxis}
   \end{tikzpicture}
   \begin{tikzpicture}[scale=0.68]
    \begin{loglogaxis}[ xlabel={time for linear solves [s]}, ylabel={$\| u - u_h \|_{L^2(\Omega^h(T))}$},
    legend entries ={ $k=2$, $k=3$, $k=4$, $k=5$, $k=6$}, legend style={anchor=north,legend columns=1, draw=none}, legend pos = south west,
     ]
     \addplot table [x index = 10, y index =1] {num_exp/out/conv_drift_DG_ks2_kt2_both_nref6_gamma0.05_K0.0_sm0_ktls1_ksls1_alpha0.0.dat};
     \addplot table [x index = 10, y index =1] {num_exp/out/conv_drift_DG_ks3_kt3_both_nref6_gamma0.05_K0.0_sm0_ktls1_ksls1_alpha0.0.dat};
     \addplot table [x index = 10, y index =1] {num_exp/out/conv_drift_DG_ks4_kt4_both_nref5_gamma0.05_K0.0_sm0_ktls1_ksls1_alpha0.0.dat};
     \addplot table [x index = 10, y index =1] {num_exp/out/conv_drift_DG_ks5_kt5_both_nref4_gamma0.05_K0.0_sm0_ktls1_ksls1_alpha0.0.dat};
     \addplot table [x index = 10, y index =1] {num_exp/out/conv_drift_DG_ks6_kt6_both_nref4_gamma0.05_K0.0_sm0_ktls1_ksls1_alpha0.0.dat};
    \end{loglogaxis}
   \end{tikzpicture}
\end{center}
\caption{Left: conservation error $e_{\mathrm{mass}}$ of the primal method for $k=k_s=k_t$ on uncurved (dashed) and curved (solid) meshes. Right: $L^2$ error at the final time against the wall-clock time of the linear solves accumulated over all time slabs, for the primal method on uncurved meshes; the markers correspond to the refinement levels $i=0,1,\dots$}
\label{plot_cost_mass}
\end{figure}

\subsection{Computational cost} \label{sect_cost}
We complement the convergence study by runtime statistics. All runs of \Cref{plot1,plot2,plot3} were carried out sequentially on a desktop computer with an Intel Core i7-13700K processor, using eight threads for the assembly and for the sparse direct solver \texttt{pardiso}. \Cref{plot_cost_mass} (right) shows, for the primal method on uncurved meshes, the $L^2$ error at the final time against the wall-clock time spent in the linear solves, i.e.\ the LU factorisation of each time-slab system by \texttt{pardiso}, which serves as preconditioner of a subsequent GMRes iteration converging within very few iterations (iterative refinement), accumulated over all $N = 2^{i+1}$ time slabs. From one level to the next, the number of unknowns per time slab grows by a factor of four and the number of time slabs doubles, so that a solver of linear complexity would yield a growth of the solve time by a factor of eight per level; we observe factors between 7 and 18, increasing with the level and the polynomial degree. For errors around $10^{-3}$, the orders are comparable in cost, whereas for more stringent accuracy requirements higher orders are considerably cheaper: An $L^2$ error of $10^{-6}$ costs 347\,s of solve time for $k=2$ on level $i=5$, but 9.8\,s for $k=4$ on level $i=2$ and 8.8\,s for $k=6$ on level $i=1$.

\subsection{Ghost penalty stability study} We present an adapted version of the numerical experiment of \Cref{sect_num_2d} to demonstrate the effectiveness of the Ghost-penalty stabilisation against small cut issues. To be specific, let $u=\cos(\pi r) \cdot \cos(\frac{\pi}{2} t)$, and $\tilde x = \exp(-\beta t)(x - \delta)$, $\vec{w} = \beta (x - \delta, y)^T$ with the corresponding manufactured right-hand side, i.e.\ the configuration of \Cref{sect_num_2d} translated by $\delta$ relative to the mesh, for some small spatial perturbation parameter $\delta$. We are interested in the $L^2$ error after one time step in the setup of $i=2$, $k_s = k_t = 4$ on uncurved meshes, $q_s = q_t = 1$, and the condition number of the matrix of that timestep in dependence of $\delta$. The results are shown in \Cref{gp_study} and clearly demonstrate the Ghost penalty's effectiveness to yield a stable and accurate numerical solution. The Ghost penalty parameter $\gamma_J$ can be chosen across a range of orders of magnitude.

\begin{figure}
 \begin{tikzpicture}[scale=0.75]
  \begin{semilogyaxis}[ name=leftplot,
    xlabel=$\delta$,
    xmin=-0.2,xmax=0.2,
    xtick distance=0.1,
    every mark/.append style={mark repeat=40, mark size=1.5pt},
    legend entries={$\gamma_J=0$, $\gamma_J=5\cdot10^{-10}$,
                    $\gamma_J=5\cdot10^{-6}$, $\gamma_J=5\cdot10^{-2}$},
    legend style={at={(1.1,1.15)}, anchor=north, legend columns=4, draw=none},
    grid=both, grid style={gray!20},
      ylabel={condition number $\kappa(A)$}]
      \addplot[Set1-A, mark=o]       table [x index=0, y index=2] {num_exp/gp_robustness/gp_sweep_p4_is2_lw_gamma_0e+00.dat};
      \addplot[Set1-B, mark=x]       table [x index=0, y index=2] {num_exp/gp_robustness/gp_sweep_p4_is2_lw_gamma_5e-10.dat};
      \addplot[Set1-C, mark=+]       table [x index=0, y index=2] {num_exp/gp_robustness/gp_sweep_p4_is2_lw_gamma_5e-06.dat};
      \addplot[Set1-D, mark=diamond] table [x index=0, y index=2] {num_exp/gp_robustness/gp_sweep_p4_is2_lw_gamma_5e-02.dat};
  \end{semilogyaxis}
    \begin{semilogyaxis}[xlabel=$\delta$,at={(leftplot.right of south east)},
        xshift=-5.3cm,   
    xmin=-0.2,xmax=0.2,
    xtick={-0.1,0,0.1,0.2},
    every mark/.append style={mark repeat=40, mark size=1.5pt},grid=both, grid style={gray!20},
      ylabel={$\|u-u_h\|_{L^2(\Omega^h(t=\Delta t))}$}, 
      yticklabel pos=right,
      ylabel style={ yshift=-2.1cm,
            at={(ticklabel* cs:0.5)},
            anchor=south,
            rotate=-180
        },
        yticklabel pos=right]
      \addplot[Set1-A, mark=o]       table [x index=0, y index=1] {num_exp/gp_robustness/gp_sweep_p4_is2_lw_gamma_0e+00.dat};
      \addplot[Set1-B, mark=x]       table [x index=0, y index=1] {num_exp/gp_robustness/gp_sweep_p4_is2_lw_gamma_5e-10.dat};
      \addplot[Set1-C, mark=+]       table [x index=0, y index=1] {num_exp/gp_robustness/gp_sweep_p4_is2_lw_gamma_5e-06.dat};
      \addplot[Set1-D, mark=diamond] table [x index=0, y index=1] {num_exp/gp_robustness/gp_sweep_p4_is2_lw_gamma_5e-02.dat};
  \end{semilogyaxis}
\end{tikzpicture}
\caption{Ghost penalty effectiveness study: Condition number (left hand side) and $L^2$ error after one time step (right hand side) demonstrate the necessity of $\gamma_J > 0$.
 For large $\gamma_J$, such as $5 \cdot 10^{-2}$ in this case, the $L^2$ error will increase in absolute terms compared against the least viable stabilisation, which maps to the theoretical estimate in \Cref{eq_gp_consist_error}. The asymptotic order of convergence for $\Delta t, h \to 0$ is optimal \cite[Fig 11 right]{HLP2022}.}
\label{gp_study}
\end{figure}

\subsection{Spatially three dimensional problem} Next, we also demonstrate the applicability of our proposed method to a spatially three dimensional problem, which is motivated by our intended real-life applications in droplets and fluids: In cylindrical coordinates, we chose a parabolic flow profile in a rectangular channel, which will deform a spherical geometry into a kite-shaped object, serving as a blueprint for a deformed droplet. We mention in passing that this is a three-dimensional variant/ adaptation of the kite geometry of \cite{HLP2022, fhreport}.

In particular, let $T= 0.5$, $\Omega = [-1.2,6] \times [-1.2,1.2]^2$, $\rho(x,y,z,t) = (10 - y^2 - z^2) \cdot t$, $\vec{w} = (\partial_t \rho, 0,0)^T$, $r = \sqrt{ (x-\rho)^2 + y^2 + z^2 }$, $\phi = r - 1$. For this example, we consider the solution $ u = \cos(\pi r)$.

In \Cref{plot_kite3d}, the computational results for the method involving uncurved meshes is displayed, in particular with $k=2$. Apart from the validation given by these, we confirm the optimal orders of convergence in a smaller study in \Cref{plot4}.
\begin{figure}
 \includegraphics[width=\textwidth]{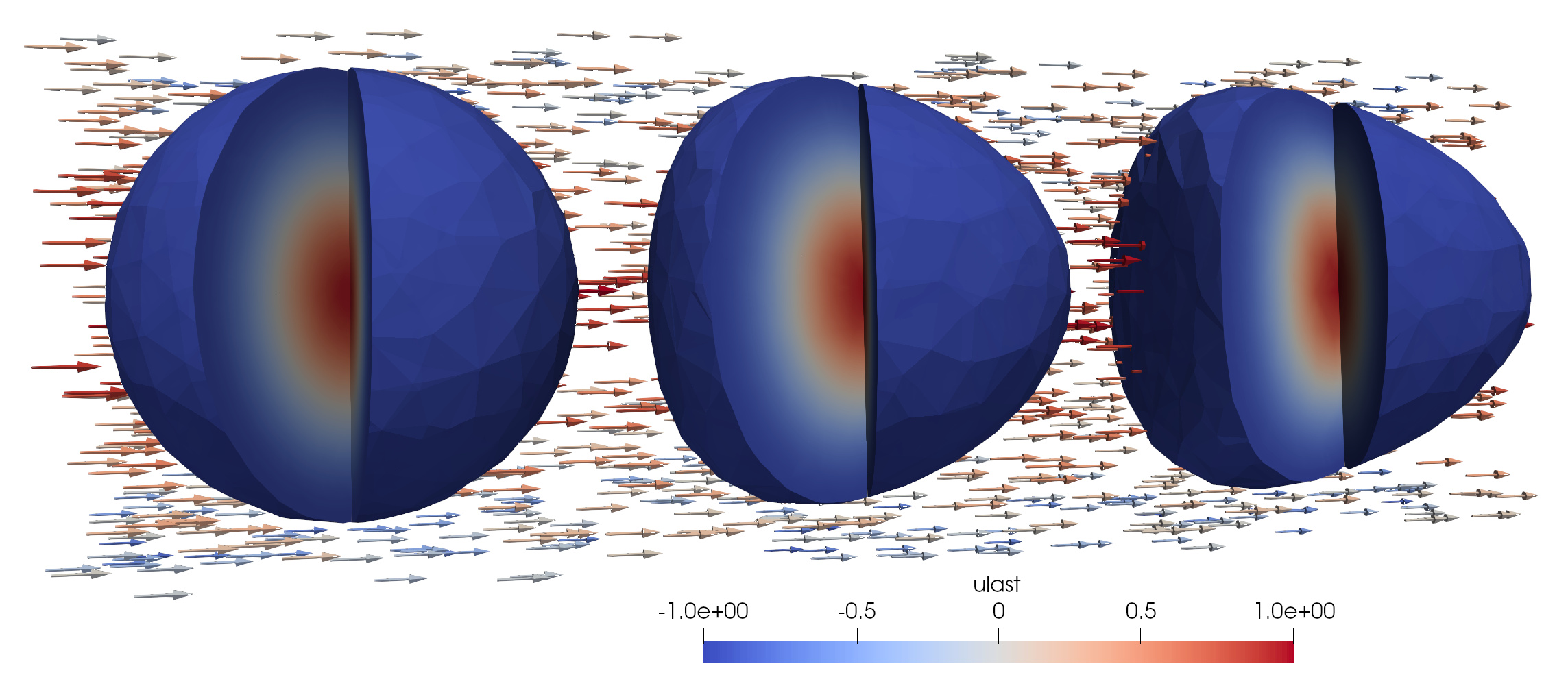}
 \caption{Flow field $\vec{w}$ (arrows) and discrete solution $u$ for three time instances between $t=0$ and $t=T$ for the spatially three dimensional problem.}
 \label{plot_kite3d}
\end{figure}
\begin{figure}
\begin{center}
   \begin{tikzpicture}[scale=0.75]
    \begin{semilogyaxis}[ xlabel=$i$, ylabel={$\| u - u_h \|_{L^2(\Omega^h(T))}$},
    legend entries ={ $k=2$}, legend style={anchor=north,legend columns=1, draw=none}, legend pos = south west,
     ]
     \addplot table [x index = 0, y index =1] {num_exp/out/conv_kite3d_DG_ks2_kt2_both_nref4_gamma0.05_K0.0_sm0_ktls1_ksls1_alpha0.0.dat};
     \addplot[gray, dashed, domain=0:3] {(1/2^(x-0.5)))^3};
    \end{semilogyaxis}
                  \node[scale=0.75] at (4.5,5.25) {$O(h^{3})= O(\Delta t^{3})$};
     \draw[scale=0.75, gray, dash=on 2.25pt off 2.25pt phase 0pt, line width=0.4*0.75pt] (6.25/0.75,5.25/0.75) -- (6.8/0.75,5.25/0.75);
   \end{tikzpicture}
   \begin{tikzpicture}[scale=0.75]
    \begin{semilogyaxis}[ xlabel=$i$, ylabel={$| u - u_h|_{H^1(Q^h)}$},
    legend entries ={ $k=2$}, legend style={anchor=north,legend columns=1, draw=none}, legend pos = south west,
     ]
     \addplot table [x index = 0, y index =2] {num_exp/out/conv_kite3d_DG_ks2_kt2_both_nref4_gamma0.05_K0.0_sm0_ktls1_ksls1_alpha0.0.dat};
     \addplot[gray, dashed, domain=1:3] {(1/2^(x-3)))^2};
    \end{semilogyaxis}
                  \node[scale=0.75] at (4.5,5.25) {$O(h^{2})= O(\Delta t^{2})$};
     \draw[scale=0.75, gray, dash=on 2.25pt off 2.25pt phase 0pt, line width=0.4*0.75pt] (6.25/0.75,5.25/0.75) -- (6.8/0.75,5.25/0.75);
   \end{tikzpicture}
\end{center}
\caption{Convergence of the discrete method for $k=k_s=k_t$ and uncurved meshes in the $L^2$ norm and the broken $H^1$ seminorm. 3D test case.}
\label{plot4}
\end{figure}

We conclude that our proposed method allows for a straightforward handling of spatially three-dimensional moving objects, paving a way towards advanced engineering applications.

\subsection*{Acknowledgments}
The authors were supported by the EPSRC, grant reference EP/X042650/1. For the purpose of open access, the authors have applied a Creative Commons Attribution (CC BY) licence to any Author Accepted Manuscript version arising.

FH was supported by the Swedish Research Council under grant no. 2021-06594 while in residence at Institut Mittag-Leffler in Djursholm, Sweden during the fall semester of 2025.

We thank two anonymous reviewers for helpful suggestions.

Moreover, we thank Stefan Frei for pointing out a correction to the scaling in \Cref{eq_B_symmetric_testing_contrib}.
\subsection*{Statement on the use of AI}
Anthropic Claude and OpenAI ChatGPT were used in the revisions of this manuscript from a previous version, particularly to assist in amending smaller technical details in the estimates, and in the discussion on the conservative formulation. The authors have carefully reviewed any incorporated input from AI and take full responsibility for the arguments put forward.
\bibliographystyle{siamplain}
\bibliography{paper}

\end{document}

%% file: st_shared_siam.tex
\usepackage{amsfonts}
\usepackage{amsmath}
\usepackage{amssymb}
\usepackage{nicefrac}
\usepackage{graphicx}
\usepackage[caption=false]{subfig}
\usepackage{epstopdf}
\usepackage{algorithmic}
\ifpdf
  \DeclareGraphicsExtensions{.eps,.pdf,.png,.jpg}
\else
  \DeclareGraphicsExtensions{.eps}
\fi

\usepackage[normalem]{ulem}

\newsiamremark{remark}{Remark}
\crefname{remark}{Remark}{Remarks}
\newsiamremark{assumption}{Assumption}
\crefname{assumption}{Assumption}{Assumptions}
\newsiamremark{hypothesis}{Hypothesis}
\crefname{hypothesis}{Hypothesis}{Hypotheses}
\newsiamthm{claim}{Claim}
\newsiamthm{conjecture}{Conjecture}
\crefname{claim}{Claim}{Claims}
\crefname{conjecture}{Conjecture}{Conjectures}
\AddToHook{env/remark/begin}{\crefalias{theorem}{remark}}
\AddToHook{env/assumption/begin}{\crefalias{theorem}{assumption}}
\AddToHook{env/hypothesis/begin}{\crefalias{theorem}{hypothesis}}
\AddToHook{env/conjecture/begin}{\crefalias{theorem}{conjecture}}
\newsiamthm{attempt}{Attempt}

\usepackage{amsopn}

\usepackage[]{hyperref}
\usepackage[capitalize]{cleveref}

\providecommand{\vertiii}[1]{{\left\vert\kern-0.25ex\left\vert\kern-0.25ex\left\vert #1
\right\vert\kern-0.25ex\right\vert\kern-0.25ex\right\vert}}

\usepackage{tikz}
\usepackage{tikz-3dplot}
\usetikzlibrary{fpu,positioning,shapes,arrows,calc, arrows.meta, patterns, intersections, decorations.pathreplacing,calligraphy}
\usetikzlibrary{shapes,arrows,positioning,fit,calc,spy, colorbrewer}

\usepackage{pgfplots}
\usepgfplotslibrary{colorbrewer,groupplots}

\pgfplotsset{
cycle list/Set1-5,
cycle multiindex* list={
mark list*\nextlist
Set1-5\nextlist
},
}

\usepackage{booktabs}

\usepackage{graphics}

\usepackage{rotating}
\usepackage{mathtools}

\providecommand{\bi}[1]{\hyperlink{def:bi}{b_{i}}}

\providecommand{\ThnG}[1]{\hyperref[eq:new_st_isoparam_regions]{\mathcal{T}_{h,#1}^{\Gamma}}}